\documentclass{jgm}
\usepackage{amsmath}
\usepackage[norelsize]{algorithm2e}
  \usepackage{paralist}
    \usepackage{psfrag}
\usepackage{color}
\usepackage{amssymb}
\usepackage{amsthm}
\usepackage{latexsym}
\usepackage{calrsfs}
  \usepackage{graphics} 
  \usepackage{epsfig} 
\usepackage{graphicx}  \usepackage{epstopdf}
\DeclareMathAlphabet{\pazocal}{OMS}{zplm}{m}{n}
\newcommand{\Ca}{\mathcal{C}}
\newcommand{\Cb}{\pazocal{C}}
\newcommand{\grad}{\nabla}

\newcommand{\Diff}{\mathcal{D}}

\def\currentvolume{9}
 \def\currentissue{2}
  \def\currentyear{2017}
   \def\currentmonth{June}
    \def\ppages{227--256}
     \def\DOI{10.3934/jgm.2017010}

\newtheorem{theorem}{Theorem}[section]

\newtheorem{lemma}[theorem]{Lemma}
\newtheorem{proposition}{Proposition}

\theoremstyle{definition}
\newtheorem{definition}[theorem]{Definition}
\newtheorem{remark}{Remark}

\title[Quotient elastic metrics on Arcspace] 
      {Quotient elastic metrics on the manifold of \\arc-length parameterized plane curves}

\author[Alice B. Tumpach  and Stephen C. Preston]{}


 \email{Barbara.Tumpach@math.univ-lille1.fr}
 \email{stephen.preston@brooklyn.cuny.edu}

\thanks{$^*$ Corresponding author: Alice B. Tumpach}

\begin{document}
\maketitle

\centerline{\scshape Alice B. Tumpach$^*$}
\medskip
{\footnotesize
 \centerline{Laboratoire Paul Painlev\'e​}
   \centerline{CNRS U.M.R. 8524}
   \centerline{59 655 Villeneuve d'Ascq Cedex, France}
} 

\medskip

\centerline{\scshape Stephen C. Preston}
\medskip
{\footnotesize
 \centerline{Department of Mathematics}
   \centerline{Brooklyn College and CUNY Graduate Center}
   \centerline{New York, USA}
}



\begin{abstract}
We study the pull-back of the $2$-parameter family of quotient elastic metrics introduced in \cite{MioAnuj} on the space of arc-length parameterized curves. This point of view has the advantage of concentrating on the manifold of arc-length parameterized curves, which is a very natural manifold when the analysis of un-parameterized curves is concerned, pushing aside the tricky quotient procedure detailed in \cite{LRK} of the preshape space of parameterized curves by  the reparameterization (semi-)group.
In order to study the problem of finding geodesics between two given arc-length parameterized curves under these quotient elastic metrics, we give a  precise computation of the gradient of the energy functional in the smooth case as well as a discretization of it, and implement a path-straightening method. This allows us to have a better understanding of how the landscape of the energy functional varies with respect to the parameters.
\end{abstract}

\section{Introduction}

The authors of \cite{MioAnuj} introduced a $2$-parameter family of Riemannian metrics $G^{a,b}$ on the space of plane curves that penalizes bending as well as stretching. The metrics within this family are now called elastic metrics. In \cite{Srivastava2011b},  it was shown that, for a certain relation between the parameters, the resulting metric is flat on parameterized open curves, whereas the space of length-one curves is the unit sphere in an Hilbert space, and the space of parameterized closed curves a codimension $2$ submanifold of a flat space. A similar method for simplifying the analysis of plane curves was introduced in \cite{Younes2008}.
These results have been generalized in \cite{Bauer}, where the authors introduced another family of metrics, including the elastic metrics as well as the metric of  \cite{Younes2008}, and studied in which cases these metrics can be described using the restrictions of flat metrics to submanifolds. In particular they showed that, for arbitrary values of the parameters $a$ and $b$, the elastic metrics $G^{a,b}$ are flat metrics on the space of parameterized open curves, and the space of parameterized closed curves a codimension $2$ submanifold of a flat space.
These results have important consequences for shape comparison and form recognition since the comparison of parameterized curves becomes a trivial task and the comparison of un-parameterized curves is greatly simplified.
In this strategy, the space of un-parameterized curves, also called \textit{shape space}, is presented as a  quotient space of the space of parameterized curves, where two parameterized curves are identified  when they differ by a reparameterization. The elastic metrics induce Riemannian metrics on shape space, called quotient elastic metrics. The remaining difficult task in comparing two un-parameterized curves under the quotient elastic metrics is to find a matching between the two curves that minimizes the distance between the corresponding reparameterization-orbits. Given this matching, computing a geodesic between two shapes is again an easy task using the flatness of the metrics.

In \cite{LRK}, a mathematically rigorous development of the quotient elastic metric used in \cite{Srivastava2011b}  is given
(i.e., with the parameters $a = \frac{1}{4}$ and $b=1$), including a careful analysis of the quotient procedure by the reparameterization semi-group.
The authors of \cite{LRK} also showed that a minimizing geodesic always exists between two curves, when at least one of them is piecewise linear.
Moreover, when both curves are piecewise linear, the minimizing geodesic can be represented by a straight line between two piecewise linear curves in the corresponding orbits. In other words the space of piecewise linear curves is a geodesically convex subset of the space of curves for the quotient elastic metric $G^{\frac{1}{4}, 1}$. Finally, in the same paper,
a precise algorithm for the matching problem of piecewise linear curves is implemented, giving a tool to compare shapes in an efficient as well as accurate manner.

In \cite{Bruveris}, it was shown that, in the same context, a minimizing geodesic for the quotient elastic metric $G^{\frac{1}{4}, 1}$ always exists between two $\mathcal{C}^1$-curves $\gamma_1$ and $\gamma_2$, meaning that there exists two elements $\phi_1$ and $\phi_2$ in the reparameterization semi-group such that the straight line between $\gamma_1 \circ \phi_1$ and $\gamma_2 \circ \phi_2$ minimizes the geodesic  distance between the orbits of $\gamma_1$ and $\gamma_2$. However, the reparameterizations $\phi_1$ and $\phi_2$ being a priori only absolutely continuous, it is not clear whether $\gamma_1 \circ \phi_1$ and $\gamma_2 \circ \phi_2$ can be chosen to be $\mathcal{C}^1$. In other words, it is  (to our knowledge) not known whether the subset of $\mathcal{C}^1$-curves is geodesically convex. In addition,  two Lipschitz-curves in the plane are constructed in  \cite{Bruveris}  for which no optimal reparameterizations exist.

In the present paper, we want to pursue another strategy for understanding the quotient elastic metrics on shape space. Indeed, instead of identifying the shape space of un-parameterized curves with a quotient space, we identify it with the space of arc-length parameterized curves. Given a shape in the plane, this consists in endowing it with the preferred parameterization by its arc-length, leading to a uniformly sampled curve. Note that any Riemannian metric on shape space can be understood as a Riemannian metric on the space of arc-length parameterized curves. In the present paper, we endow the space of arc-length parameterized curves with the quotient elastic metrics. In  \cite{KSMJ}, the manifold of arc-length parameterized curves was also studied, but the metrics used there are not the elastic ones. In \cite{Preston}, the second author studied a similar metric and its shape geometry as identified with arc-length parameterized curves; however the  computation in Theorem 6.4 of \cite{Preston} is incorrect since the horizontal space is not computed correctly.

The present paper is organized as follows. In Section~\ref{setup}, we introduce the notation used in the present paper, as well as the manifolds of curves under interest. In Section~\ref{smooth}, we concentrate on the smooth case, and compute the gradient of the energy functional associated to the quotient elastic metrics~$G^{a, b}$. In Section~\ref{discretization}, we consider  a discretization of the smooth case. This is an unavoidable step towards implementation, where each smooth curve is approximated by polygonal lines, and each smooth parameterized curve is approximated by a piecewise linear curve. Finally, in Section~\ref{implementation}, an algorithm for the two-boundary problem is presented, and some properties of the energy landscape depending on the parameters are studied.

\section{Mathematical setup}\label{setup}

\subsection{Manifolds of based parameterized curves}\label{plane_curves}

In this section, we define the manifolds of plane curves that we will consider in the present paper. First some motivation.
Roughly speaking, shape space consists of the set of curves
in the plane.
The difficulty is that although this space should be an infinite-dimensional manifold, it does not have convenient coordinate
charts. The typical approach is to consider all \emph{parameterized} curves $\gamma\colon [0, 1]\to\mathbb{R}^2$ (resp. $\gamma\colon \mathbb{S}^1\to\mathbb{R}^2$ for closed curves), which is a
linear space and hence a manifold, then consider the open subset consisting of free immersions or embeddings, then mod out by
the group of diffeomorphisms of $[0, 1]$ (resp. $\mathbb{S}^1$) which represent the reparameterizations
of a given curve (all of which correspond to the same shape). Here and in the rest of the paper $\mathbb{S}^1$ will denote the circle of length one given by $$\mathbb{S}^1 = \mathbb{R}/\mathbb{Z}.$$
This quotient space admits a structure of smooth Fr\'echet manifold (see Theorem~1.5 in \cite{CMM91} for a detailed construction of the coordinate
charts in the smooth category), and the set of free immersions or embeddings is a principal bundle over this quotient space with structure group the group of diffeomorphisms (see \cite{BBM} for an overview of the theory). In this paper, we will identify this quotient space with the space of arc-length parameterized curves, which is a nice submanifold of the space of parameterized curves (see Theorem~\ref{arcspace} and Theorem~\ref{diffeoquotient} below). See also section~3.1. in \cite{Brylinski}, where an analogous construction is carried out for loops in $\mathbb{R}^3$ and where the K\"ahler structure of these loop spaces is explained.
Let us stress some choices we made:
\begin{itemize}
\item We will work with \emph{based} oriented curves (that is, with a specified start and endpoint) rather than closed curves; the advantage of this is that we have a unique constant-speed parameterization. It is also closer to the implementation, where a curve is replaced by a finite number of points, which are stored in a matrix and indexed from $1$ to $n$. In our applications later the curves will all happen to be closed, but the analysis will be independent of the choice of base point (i.e., of the ordering of the points).
\item We get a further simplification by restricting to those curves of total length one; then we get a \emph{unit}-speed parameterization, and we do not have to carry the length around as an extra parameter.
\item We will work with immersions rather than embeddings since the embedding constraint is somewhat tricky to enforce.
\item Finally since our Riemannian metric (defined in the next section) will depend only on the derivative $\gamma'$, we shall identify all curves up to translation,
which is of course equivalent with simply working with $\gamma'$ rather than $\gamma$, where $\gamma'$ has to satisfy $\int_0^1\gamma'(s) ds = 0$ for closed curves.
\end{itemize}
In this section, $I = [0, 1]$ (for open curves) or $I = \mathbb{S}^1 = \mathbb{R}/\mathbb{Z}$ (for closed curves).

\subsubsection{Curves modulo translations}\label{modtranslation}
Let $k\in\mathbb{N}$, and define a norm on the vector space $\Ca^{k}(I, \mathbb{R}^2)$ of differentiable curves of order $k$, $\gamma\colon I\to\mathbb{R}^2$, by

\ \vspace*{-10pt}
\begin{equation}\label{extratopology}
\lVert \gamma\rVert_{\Ca^k} := \sum_{j=1}^k  \max_{s\in I} \lvert \gamma^{(j)}(s) \rvert,
\end{equation}
where, for $z\in\mathbb{R}^2$, $\lvert z\rvert$ denotes the norm of $z$. 
 The purpose of starting the first sum at $j=1$ instead of $j=0$ is to reduce to the quotient space by translations $\Ca^{k}(I, \mathbb{R}^2)/\mathbb{R}^2$, so that only $\gamma'$ matters. This corresponds to considering curves in $\mathbb{R}^2$ irrespective of their positions in comparison to the origin of $\mathbb{R}^2$. The quotient vector space $\Ca^{k}(I, \mathbb{R}^2)/\mathbb{R}^2$ endowed with the norm induced by \eqref{extratopology} is a Banach space.
We could identify it with any complement to the subspace of constant functions, for instance with the subspace $\Ca^{k}_{cen}(I, \mathbb{R}^2)$ of centered curves (i.e., curves whose center of mass lies at the origin of $\mathbb{R}^2$)
\begin{equation}\label{centered}
\Ca^{k}_{cen}(I, \mathbb{R}^2) = \left\{\gamma\in \Ca^{k}(I, \mathbb{R}^2), \int_0^1\gamma(s) ds = 0\right\},
\end{equation}
or with the subspace $\Ca^{k}_0(I, \mathbb{R}^2)$ of curves starting at $z = 0$
\begin{equation}\label{start0}
\Ca^{k}_0(I, \mathbb{R}^2) = \left\{\gamma\in \Ca^{k}(I, \mathbb{R}^2), \gamma(0)  = 0\right\},
\end{equation}
which are Banach spaces for the norm \eqref{extratopology}. Despite the fact that the identification of the quotient space $\Ca^{k}(I, \mathbb{R}^2)/\mathbb{R}^2$ with a complement to $\mathbb{R}^2$ in
$\Ca^{k}(I, \mathbb{R}^2)$ may seem natural in theory, it introduces unnecessary additional constraints as soon as numerics are involved: indeed restricting ourselves to centered curves implies that the tangent space to a curve contains only centered vector fields, i.e., vector fields $Z$ along the curve which preserve condition~\eqref{centered}, i.e., such that $\int_0^1 Z(s) ds = 0$, and for curves starting at the origin we get the constraint $Z(0) = 0$.
Since the elastic metrics introduced in the next section are degenerate in the direction of translations, the distance between two curves $\gamma_1$ and $\gamma_2$ will match the distances between $\gamma_1+c_1$ and $\gamma_2 + c_2$ for any constants $c_1$ and $c_2$. This degeneracy property implies that in the numerics, we can freely choose how to represent a curves modulo translation. Depending on what we want to emphasize, one may prefer the centered curves or the curves starting at the origin.
\subsubsection{Smooth immersions}
Recall that $\gamma\colon I\to\mathbb{R}^2$ is an immersion if and only if $\gamma'(s)\neq 0$ for all $s\in I$.
In the topology given by the norm \eqref{extratopology}, the set of all  $\Ca^k$-immersions is an open
subset of the Banach space $\Ca^{k}(I, \mathbb{R}^2)/\mathbb{R}^2$, hence a Banach submanifold of $\Ca^{k}(I, \mathbb{R}^2)/\mathbb{R}^2$. It is denoted by $\Cb^k(I)$:
$$
\Cb^k(I) = \left\{\gamma \in \Ca^{k}(I, \mathbb{R}^2)/\mathbb{R}^2, \gamma'(s)\neq 0, \forall s \in I\right\}.
$$
 The vector space $\Ca^{\infty}(I, \mathbb{R}^2)/\mathbb{R}^2 = \cap_{k=1}^\infty \Ca^{k}(I, \mathbb{R}^2)/\mathbb{R}^2$ of smooth curves $\gamma\colon I\to\mathbb{R}^2$ modulo translations  endowed with the family of norms
$\lVert \cdot\rVert_{\Ca^k}$ is a graded Fr\'echet space (see Definition~II.1.1.1 in \cite{Ham82}). The space of smooth immersions
\begin{equation}\label{K0def}
\Cb(I) = \bigcap_{k=1}^{\infty}\Cb^k(I) = \{\gamma \in \Ca^{\infty}(I, \mathbb{R}^2)/\mathbb{R}^2, \gamma'(s)\neq 0, \forall s \in I\}.
\end{equation}
is an open set of $\mathcal{C}^{\infty}(I, \mathbb{R}^2)/\mathbb{R}^2$ for the topology induced by the family of norms $\lVert \cdot\rVert_{\Ca^k}$, hence a Fr\'echet manifold.
\begin{remark}
In the space of smooth immersions $\Cb([0, 1])$, we can consider the subset of curves $\gamma$  which are closed, i.e., such that $\gamma(0) = \gamma(1)$, or equivalently such that $\int_0^1\gamma'(s) ds = 0$. Let us denote it by $\Cb_c([0, 1])$. Then $\Cb(\mathbb{S}^1)\subset \Cb_c([0, 1])$. Indeed a curve $\gamma\in \Cb(\mathbb{S}^1)$ has all its derivatives matching at $0$ and $1$, whereas a curve in $\Cb_c([0, 1])$ may have a failing in smoothness at $0$. Note that $\mathcal{C}^{\infty}(\mathbb{S}^1, \mathbb{R}^2)$ is a closed subset in $\mathcal{C}^{\infty}([0,1], \mathbb{R}^2)$ which is not a direct summand (see Example 1.2.2 in \cite{Ham82}). Moreover note that the derivative which maps $\gamma$ to $\gamma'$ from $\mathcal{C}^{\infty}(I, \mathbb{R}^2)/\mathbb{R}^2$ into $\mathcal{C}^{\infty}(I, \mathbb{R}^2)$ is onto for open curves, but has range equal to the closed subspace $\{f\in \mathcal{C}^{\infty}(\mathbb{S}^1, \mathbb{R}^2), \int_0^1f(s)\,ds = 0\}$ for closed curves.
\end{remark}

\subsubsection{Length-one curves}
We denote the subset of length-one immersions modulo translations by
\begin{equation}\label{K1def}
\Cb_1(I) = \{ \gamma \in \Cb(I) : \int_0^1 \lvert \gamma'(s)\rvert \, ds = 1\}.
\end{equation}
Recall that the implicit function theorem is invalid for general Fr\'echet manifolds, but is valid in the category of tame Fr\'echet manifolds and tame smooth maps,
and is known as the implicit function theorem of Nash-Moser (Theorem~III.2.3.1 of \cite{Ham82}, page 196). Recall that a linear map $A\colon F_1\rightarrow F_2$ between graded Fr\'echet spaces is
\emph{tame} if there exists some $r$ and $b$ such that $\|Af\|_{n}\leq C_n\|f\|_{n+r}$ for each $n\geq b$ and some constants $C_n$ (see Definition~II.1.2.1 page 135 in \cite{Ham82}).
A Fr\'echet space is \emph{tame} if it is a tame direct summand in a space $\Sigma(B)$ of exponentially decreasing sequences in some Banach space $B$.
A nonlinear map $P$ from an open set $U$ of a graded Fr\'echet space $F_1$
into another graded Fr\'echet space $F_2$ is \emph{tame} if it is continuous and if there exists $r$ and $b$ such that
$$\|P(f)\|_n \leq C_n(1 + \|f\|_{n+r})$$ for each $n\geq b$ and some constants $C_n$ (see Definition II.2.1.1. page 140 in \cite{Ham82}).
A \emph{tame Fr\'echet manifold} is a manifold modelled on a tame Fr\'echet space, such that all transition functions are tame.

\begin{proposition}\label{lengthoneprop}
The subset $\Cb_1(I)$ of length-one immersions modulo translations defined by \eqref{K1def} is a tame $\Ca^{\infty}$-submanifold of the tame Fr\'echet manifold $\Cb(I)$ of immersions modulo translations defined by \eqref{K0def}
for the Fr\'echet manifold structure induced by the family of norms given in \eqref{extratopology}.
\end{proposition}

\begin{proof}
As an open set of $\mathcal{C}^{\infty}(I, \mathbb{R}^2)/\mathbb{R}^2$,  $\Cb(I)$ is a manifold with only one chart, hence a $\Ca^\infty$-manifold. Moreover,
$\Cb(I)$ is a tame Fr\'echet manifold in the sense of Definition~II.2.3 in \cite{Ham82}. To see this,  first note that by Theorem~II.1.3.6 page 137 in \cite{Ham82}, $\Ca^\infty(I,\mathbb{R})$ is tame since $I$ is compact. Moreover  by Lemma~II.1.3.4. page 136, the Cartesian product of two tame spaces is tame. It follows that $\Ca^\infty(I,\mathbb{R}^2)$ is a tame Fr\'echet space. By Lemma~II.1.3.3 in \cite{Ham82}, the subspace
$\mathcal{C}^{\infty}_0(I, \mathbb{R}^2)$ is also tame because its complement is one-dimensional and any map from a tame Fr\'echet space into a finite dimensional space is tame.
Since the quotient $\mathcal{C}^{\infty}(I, \mathbb{R}^2)/\mathbb{R}^2$ is isomorphic as a Fr\'echet space to $\mathcal{C}^{\infty}_0(I, \mathbb{R}^2)$, it is also tame. Hence $\Cb(I)$
is modelled on a tame Fr\'echet space and since there is only one transition function which is the identity hence tame, $\Cb(I)$ is a tame Fr\'echet manifold. Let us endow it with the complete atlas
consistent with this $\Ca^\infty$ tame manifold structure. In particular, the following coordinate charts, as used in \cite{Preston}, belong to the atlas: for each
$\gamma\in \Cb$ we write
\begin{equation}\label{gammacoords}
\gamma'(s) = e^{\sigma(s)} \big( \cos{\theta(s)}, \sin{\theta(s)}\big) =  e^{\sigma(s) + i\theta(s)},
\end{equation}
where $\sigma \in \Ca^\infty(I,\mathbb{R})$ and $\theta\in \Ca^\infty(I, \mathbb{R})$.
We get a diffeomorphism from the open set
$$
\left\{(\sigma, \theta)\in\Ca^\infty(I,\mathbb{R})\times \Ca^\infty(I, \mathbb{R}), \theta(0)\in]\theta_0+2\pi n, \theta_0+2\pi(n+1)[\right\}
$$
of the Fr\'echet space $\Ca^\infty(I,\mathbb{R})\times \Ca^\infty(I, \mathbb{R})$
onto the open subset of $\Cb(I)$ consisting of those curves such that $\frac{\gamma'(0)}{|\gamma'(0)|}\ne e^{i\theta_0}$.
The coordinate transition functions
are easily seen to be the identity in the first component (since $\rho$ is uniquely determined) and horizontal translations in the
second component, hence are clearly tame.

In $(\sigma, \theta)$-coordinates the condition \eqref{K1def} is described by the condition $L(\sigma) = 1$, where $$L(\sigma) = \int_0^1 e^{\sigma(s)} \,ds.$$ Hence $\Cb_1(I)$ is the
inverse image of a real function that is obviously $\Ca^{\infty}$. The derivative of $L$ with respect to $(\sigma, \theta)$-coordinates
may be expressed as
$$DL_{(\sigma,\theta)}(\rho,\phi) = \frac{\partial}{\partial t}\Big|_{t=0} L(\sigma+t\rho, \theta+t\phi) = \int_0^1 \rho(s) e^{\sigma(s)} \, ds;$$
the kernel of this map splits at any $(\sigma,\theta)\in L^{-1}(1)$ since we can write
$$ (\rho, \phi) = \Big( \rho - C, \phi\Big) + \Big(C, 0\Big), \qquad C = \int_0^1 \rho(x) e^{\sigma(x)} \, dx,$$
where $\Big( \rho - C, \phi\Big)$ belongs to the kernel of $DL_{(\sigma,\theta)}$, which is closed, and $\Big(C, 0\Big)$ belongs to a one-dimensional subspace of $\Ca^\infty(I,\mathbb{R})\times \Ca^\infty(I, \mathbb{R})$, which is therefore also closed. Since the image of $(C,0)$ is obviously $C$, the derivative is also
surjective.

By the implicit function theorem of Nash-Moser (Theorem~III.2.3.1 of \cite{Ham82}, page 196), $\Cb_1(I)$ is a smooth tame submanifold of $\Cb(I)$.
\end{proof}

\subsubsection{Arc-length parameterized curves}
Now we consider the space $\mathcal{A}_1(I)$ of arc-length parameterized curves on $I$ modulo translations:
\begin{equation}\label{arcspacedef}
\mathcal{A}_1(I) = \{\gamma\in \Cb(I)~: \lvert \gamma'(s)\rvert = 1, \;\forall s\in I\}.
\end{equation}
Obviously $\mathcal{A}_1(I)\subset \Cb_1(I)$.

\begin{theorem}\label{arcspace}
The space $\mathcal{A}_1(I)$ of arc-length parameterized curves on $I$ modulo translations defined by \eqref{arcspacedef} is a tame $\Ca^{\infty}$-submanifold of $\Cb(I)$,
and thus also of $\Cb_1(I)$. Its tangent space at a curve $\gamma$ is
$$
T_{\gamma}\mathcal{A}_1 = \{w \in \mathcal{C}^{\infty}(\mathbb{S}^1, \mathbb{R}^2), w'(s)\cdot \gamma'(s) = 0, \quad \forall s\in \mathbb{S}^1\}.
$$
\end{theorem}

\begin{proof}
The proof is very simple: the space $\mathcal{A}_1(I)$ is closed and looks, in any $(\sigma, \theta)$-coordinate chart,
like $\{(\sigma, \theta): \sigma\equiv 0\}$, which is just the definition of a submanifold. Since the $(\sigma, \theta)$-coordinate charts are tame, $\mathcal{A}_1(I)$ is a tame submanifold of $\Cb(I)$.
The fact that $\mathcal{A}_1(I)$ is also a smooth Fr\'echet tame submanifold of $\Cb_1(I)$ follows from the universal
mapping property of submanifolds. The expression of the tangent space is straightforward.
\end{proof}

\subsubsection{Reparameterizations of curves}
Reparameterizations of open curves are given by smooth diffeomorphisms
$\phi\in \Diff^+([0,1])$, the plus sign denoting that these
diffeomorphisms preserve $0$ and $1$.
For closed curves, we will denote by $\Diff^+(\mathbb{S}^1)$ the group of diffeomorphisms of $\mathbb{S}^1$  preserving the orientation. In the following we will denote by $\mathcal{G}(I)$ either the group $\Diff^+([0,1])$ when considering open curves (i.e., when $I = [0, 1]$), or $\Diff^+(\mathbb{S}^1)$ for closed curves (i.e., when $I = \mathbb{S}^1 = \mathbb{R}/\mathbb{Z}$).
By Theorem~II.2.3.5 
in \cite{Ham82}, $\mathcal{G}(I)$ is  a tame Fr\'echet Lie group.

\begin{proposition}
The right action $\Gamma\colon \Cb(I)\times\mathcal{G}(I) \rightarrow \Cb(I)$, $\Gamma(\gamma,\psi)=\gamma\circ\psi$ of the group of reparameterizations $\mathcal{G}(I)$ on the tame Fr\'echet manifold
$\Cb(I)$ is smooth and tame, and preserves $\Cb_1(I)$.
\end{proposition}

\begin{proof}
Note that the action $\Gamma$ of $\mathcal{G}(I)$ on $\Cb(I)$ is continuous  for the Fr\'echet manifold structure on $\Cb(I)$ since
$$\lVert \gamma_1\circ\phi - \gamma_2\circ\phi\rVert_{\Ca^k} = \sum_{j=1}^k \max_{s\in I} \Big\lvert \frac{d^j}{ds^j}\gamma_1(\phi(s))-\frac{d^j}{ds^j}\gamma_2(\phi(s))\Big\rvert $$
can be bounded by the chain rule in terms of $\lVert \gamma_1-\gamma_2\rVert_{\Ca^k}$ and $\lVert \phi\rVert_{\Ca^k}$. It follows that $\Gamma$ is tame.
Moreover the action   of $\mathcal{G}(I)$ on $\Cb(I)$ is differentiable: considering a family $\phi(t,s)\in\mathcal{G}(I)$ and $\gamma(t,s)\in\Cb(I)$ with
$\phi_t(0,s) = \zeta(s)$ in the Lie algebra $\textrm{Lie}(\mathcal{G}(I))$ of $\mathcal{G}(I)$, and $\gamma_t(0,s)=w(s)\in \mathcal{C}^{\infty}(I, \mathbb{R}^2)/\mathbb{R}^2$.
The derivative of the action $\Gamma:=(\gamma,\phi)\mapsto\gamma\circ\phi$ is
\begin{equation}\label{Gammaderivative}
\begin{split}
(D\Gamma)_{(\gamma,\phi)}(w,\zeta) &= \frac{\partial}{\partial t}\Big|_{t=0} \gamma\big(t,\phi(t,s)\big) =
\gamma_t(t,\phi(t,s)) + \gamma_s(t,\phi(t,s)) \phi_t(t,s)\Big|_{t=0} \\
&= w(\phi(s)) + \gamma'(\phi(s)) \zeta(s).
\end{split}
\end{equation}
Since the map which assigns $\gamma\in\Cb(I)$ to $\gamma'\in\Cb(I)$ satisfies
$\|\gamma'\|_n \leq \|\gamma\|_{n+1}$,  it is a tame linear map (with $r = 1$ and $b = 1$), continuous for the Fr\'echet manifold structure on $\Cb(I)$.
Hence $D\Gamma$ is continuous as a map  from a neighborhood of $(\gamma, \phi)$ in $\mathcal{G}(I)\times\Cb(I)$ times the Fr\'echet space
$\textrm{Lie}(\mathcal{G}(I))\times \mathcal{C}^{\infty}(I, \mathbb{R}^2)/\mathbb{R}^2$ into $\mathcal{C}^{\infty}(I, \mathbb{R}^2)/\mathbb{R}^2$, and tame.
More generally, the $k$th derivative of the action $\Gamma$ will involve only a finite number of derivatives of the curve $\gamma$, hence will be continuous and tame.
\end{proof}

\subsubsection{Quotient spaces}
Recall that an immersion $\gamma\colon I\rightarrow \mathbb{R}^2$ is free if and only if the group of reparameterizations $\mathcal{G}(I)$ acts freely on $\gamma$, i.e., the only diffeomorphism $\psi$ satisfying $\gamma\circ\psi = \gamma$ is the identity.
By Lemma 1.3 in \cite{CMM91}, a diffeomorphism having a fixed point and stabilizing a given immersion is necessarily equal to the identity map. Hence for open curves, every smooth immersion is free, since any diffeomorphism in $\mathcal{D}^+([0, 1])$ fixes $0$ and $1$.
For closed curves, the set of free immersions is an open set in the space of immersions (see \cite{CMM91}, section~1).   We will denote it by $\Cb^{f}(I)$. Note that since $I$ is compact, any $f\in \Cb(I)$ is proper. Recall the following theorem in \cite{CMM91}:
\begin{theorem}{(Theorem~1.5 in \cite{CMM91})}\label{quotient}
The quotient space $\Cb^{f}(I)/\mathcal{G}(I)$ of free immersions by the group of diffeomorphisms $\mathcal{G}(I)$ admits a Fr\'echet manifold structure such that the canonical projection $\pi\colon \Cb^{f}(I)\rightarrow \Cb^{f}(I)/\mathcal{G}(I)$ defines a smooth principal bundle with structure group $\mathcal{G}(I)$.
\end{theorem}
\begin{remark}{\rm Since $\mathcal{G}(I)$ stabilizes the submanifold $\Cb_1(I)$ of length-one curves, the quotient $\Cb_1^{f}(I)/\mathcal{G}(I)$ inherits a Fr\'echet manifold structure such that $\Cb_1^{f}(I)/\mathcal{G}(I)$ is a submanifold of $\Cb^{f}(I)/\mathcal{G}(I)$. See also \cite{TDiez} for a new slice theorem in the context of tame Fr\'echet group actions.
}
\end{remark}

\subsubsection{Orbits under the group of reparameterizations}
The orbit of $\gamma\in\Cb_1(I)$ with respect to the action by reparameterization will be denoted by  $$\mathcal{O} = \{\gamma\circ \phi \, \vert\, \phi \in \mathcal{G}(I)\}.$$
The tangent space to the orbit $\mathcal{O}$ at $\gamma\in \Cb_1(I)$ is the space of tangent vector fields along $\gamma$ (preserving the start and end points when the curve is open), i.e., the space of vector fields which are, for each value of the parameter $s\in I$, collinear to the unit tangent vector $\operatorname{v}(s) = \frac{\gamma'(s)}{\lvert \gamma'(s)\rvert}$. Such a vector field can be written $w(s) = m(s) \operatorname{v}(s)$, where $m$  is a real function corresponding to the magnitude of $w$ and such that:
\begin{itemize}
\item $m\in\mathcal{C}^{\infty}([0,1], \mathbb{R})$ satisfies $m(0) = 0$ and $m(1)=0$ for open curves,
\item  $m \in \mathcal{C}^{\infty}(\mathbb{S}^1, \mathbb{R})$ for closed curves, in particular $m(0) = m(1)$ and $m'(0) = m'(1)$.
\end{itemize}

\subsubsection{Projection on the space of arc-length parameterized curves}  Any smooth curve in the plane admits a unique reparameterization by its arc-length. This property singles out a preferred parameterized curve in the orbit of a given parameterized curve under the group of reparameterizations.
\begin{theorem}\label{diffeoCA}
Given a curve $\gamma\in \Cb_1(I)$, let $p(\gamma) \in \mathcal{A}_1(I)$ denote its arc-length-reparameterization, so that
$p(\gamma) = \gamma\circ\psi$ where
\begin{equation}\label{reparamdiff}
\psi'(s) = \frac{1}{\lvert \gamma'\big(\psi(s)\big)\rvert}, \qquad \psi(0)=0.
\end{equation}
Then $p$ is a smooth retraction of $\Cb_1(I)$ onto $\mathcal{A}_1(I)$.
\end{theorem}

\begin{proof}
The definition of $\psi$ comes from the requirement that
$\lvert (\gamma\circ\psi)'(s)\rvert = 1$, which translates into $\lvert \gamma'(\psi(s))\rvert \psi'(s)=1$.
The additional requirement $\psi(0)=0$ gives a unique solution. It is not obvious from here
that $\psi(1)=1$, but this is easier to see if we let $\xi$ be its inverse; then $\xi'(t) = \lvert \gamma'(t)\rvert$, and since $\gamma$
has length one and $\xi(0)=0$ we know $\xi(1)=1$; thus also $\psi(1)=1$. The image of this map is of course in $\mathcal{A}_1(I)$.
Smoothness follows from the fact that $\psi$ depends smoothly on parameters as the solution of an ordinary differential equation,
together with smoothness of the right action $\Gamma(\gamma,\psi)=\gamma\circ\psi$.
The fact that $p$ is a retraction follows from the obvious fact that if $\lvert \gamma'(s)\rvert \equiv 1$, then the unique solution of
\eqref{reparamdiff} is $\phi(s)=s$, so that $p\vert_{\mathcal{A}_1(I)}$ is the identity.
\end{proof}

\subsubsection{Identification of the quotient space with the space of arc-length parameterized curves}
The identification of the quotient space $\Cb_1([0, 1])/\mathcal{G}([0, 1])$ with the space $\mathcal{A}_1([0, 1])$ of arc-length parameterized curves relies on the fact that given a parameterized curve there is a \emph{unique} diffeomorphism fixing the start and endpoints which maps it to an arc-length parameterized curve.

\begin{theorem}\label{diffeoquotient}
The Fr\'echet manifold $\mathcal{A}_1([0,1])$ is diffeomorphic to the quotient Fr\'echet manifold $\Cb_1([0,1])/\mathcal{G}([0,1])$.
\end{theorem}

\begin{proof}
Since $p(\gamma\circ\psi) = p(\gamma)$ for any reparameterization $\psi\in\ \mathcal{G}([0, 1])$,
we get a smooth map $$\widetilde{p}\colon \Cb_1([0,1])/\mathcal{G}([0,1]) \to \mathcal{A}_1([0,1]),$$ which is clearly a bijection, and its inverse is
$\pi\circ\iota$ where $\pi$ is the quotient projection and $\iota$ is the smooth inclusion of $\mathcal{A}_1([0,1])$ into $\Cb_1([0,1])$.
\end{proof}
For closed curves, the subgroup $\mathbb{S}^1$ of $\mathcal{G}(\mathbb{S}^1)$ acts on a closed curve $\gamma$ by translating the base point along the curve: $\gamma(s)\mapsto\gamma(s+\tau)$ for $\tau\in \mathbb{S}^1$.
One has the following commutative diagram, where the vertical lines are the canonical projections on the quotients spaces.
$$
\begin{array}{cccc}
p\colon & \Cb_1(\mathbb{S}^1) & \longrightarrow & \mathcal{A}_1(\mathbb{S}^1)\\
 & \downarrow & & \downarrow\\
 & \Cb_1(\mathbb{S}^1)/\mathcal{G}(\mathbb{S}^1) &\longrightarrow  &  \mathcal{A}_1(\mathbb{S}^1)/\mathbb{S}^1
\end{array}
$$

\begin{figure}[!ht]
 		\centering
				\includegraphics[width = 12cm]{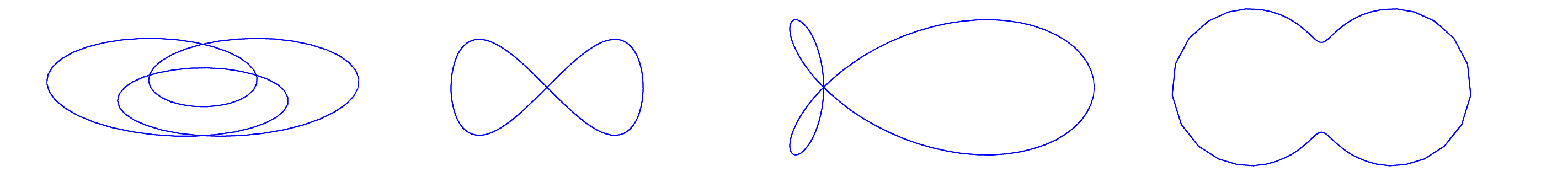}
 		\caption{Some parameterized closed immersions $\gamma$  in the plane. 
        }
 		\label{curves}		
\end{figure}

\section{Quotient elastic metrics on smooth arc-length parameterized plane curves}\label{smooth}

\subsection{Definition of the elastic metrics}\label{elastic_metric}
For $I = [0, 1]$ or $I = \mathbb{S}^1=\mathbb{R}/\mathbb{Z}$,
we will consider the following $2$-parameter family of metrics on the space $\Cb_1(I)$ of plane curves:
  \begin{equation}\label{elastic_metric_abc}
  \begin{array}{l}
  G^{a,b}(w,w) =
 \int_0^1 \left(a \left(D_s w\cdot \operatorname{v}\right)^2 + b \left(D_s w, \operatorname{n}\right)^2  \right) \lvert\gamma'(t)\rvert \, dt,
  \end{array}
  \end{equation}
where  $a$ and $b$ are positive constants, $\gamma$ is any parameterized curve in $\Cb_1(I)$, $w$ is any element of the tangent space $T_\gamma\Cb_1(I)$, with $D_s w = \frac{w'}{\lvert\gamma'\rvert}$ denoting the arc-length derivative of $w$, $\operatorname{v}=\gamma'/\lvert \gamma'\rvert$ and $\operatorname{n}=\operatorname{v}^{\perp}$. These metrics  have been introduced in
\cite{MioAnuj},  and are now called  elastic  metrics. They have been also studied in~\cite{Bauer} with another convention for the coefficients ($a$ in~\cite{MioAnuj} equals $b^2$ in~\cite{Bauer}, and $b$ in~\cite{MioAnuj} equals  $a^2$ in~\cite{Bauer}).
For $w_1$ and $w_2$ two tangent vectors at $\gamma\in \Cb_1(I)$, the corresponding inner product reads:
\begin{equation}
  \begin{array}{l}
\!\!\!  G^{a,b}\!(w_1,w_2) \!= \!\!
 \int_0^1 \!\Big(a\left(D_s w_1 \cdot \operatorname{v}\right)\!\left( D_s w_2\cdot  \operatorname{v}\right) \!+\! b \left( D_s w_1\cdot \operatorname{n}\right)\!\left(D_s w_2\cdot \operatorname{n}\right)\Big)\lvert\gamma'(t)\rvert\,dt.
  \end{array}
  \end{equation}

 The metric $G^{a,b}$ is
  invariant with respect to the action of the reparameterization group $\mathcal{G}(I)$ on $\Cb_1(I)$ and therefore it defines a metric on the quotient space $\Cb^f_1(I)/\mathcal{G}(I)$, which we will refer to as the \textit{quotient elastic metric}.

  \subsection{Horizontal space for the elastic metrics}
  Let us now consider an initial curve $\gamma$ located on the submanifold $\mathcal{A}_1(I)$ of curves parameterized by arc-length and of length 1. Recall that in this case, one has $\lvert\gamma'(s)\rvert =1$ and $D_s = \frac{d}{ds}$.
Any tangent vector  $u\in T_{\gamma}\mathcal{O}$ at $\gamma\in\mathcal{A}_1(I)$ can be written as $u(t) = m(t) \operatorname{v}(t)$ where $m\in\mathcal{C}^{\infty}([0, 1], \mathbb{R})$ satisfies $m(0) = 0$ and $m(1)=0$ for open curves and $m\in \mathcal{C}^{\infty}(\mathbb{S}^1, \mathbb{R})$ for closed curves.
The orthogonal space to $T_\gamma\mathcal{O}$ for the elastic metric $G^{a, b}$ on $\Cb_1(I)$ is called the \textit{horizontal space} at $\gamma$.
 \begin{proposition}
 The horizontal space $\textrm{Hor}$ at $\gamma\in\mathcal{A}_1(I)$ is
 \begin{equation}
\begin{array}{l}
 \textrm{Hor}_{\gamma} = \left\{w \in  T_{\gamma}\Cb_1(I),
 \left(w'\cdot \operatorname{v}\right)' =  \frac{b}{a} \kappa \left( w'\cdot \operatorname{n}\right)\right\}.
\end{array}
\end{equation}
 \end{proposition}

 \begin{proof}
 Let $u = m\operatorname{v}\in T_{\gamma}\mathcal{O}$. One has:
 \begin{equation*}
 \begin{array}{lr}
u'\cdot \operatorname{v}= m'(s),&\quad
 u'\cdot \operatorname{n} = m(s) \kappa(s). \end{array}
 \end{equation*}
The horizontal space at $\gamma$ consists of vector fields $w\in T_{\gamma}\Cb_1(I)$  such that for any function $m\in \mathcal{C}^{\infty}(I,\mathbb{R})$ (with $m(0) =m(1) = 0$ for open curves), the following quantity vanishes:
 \begin{equation*}
  \begin{array}{l}
 0 = G^{a,b}(w, m \operatorname{v}) = \int_0^1 \left(a m'(s)\left(w'(s)\cdot \operatorname{v}(s)\right) + b m(s) \kappa(s) \left( w'(s)\cdot \operatorname{n}(s)\right) \right) ds.
  \end{array}
  \end{equation*}
 After integrating the first term by parts, one obtains the following condition on $w$, which has to be satisfied for any real function $m\in\mathcal{C}^{\infty}(I,\mathbb{R})$ (with $m(0) = 0$ and $m(1)=0$ for open curves):
 \begin{equation*}
  \begin{array}{l}
0 =  \int_0^1 m\left(-a \left(w'\cdot \operatorname{v}\right)' + b \kappa \left( w'\cdot \operatorname{n}\right) \right) ds.
  \end{array}
\end{equation*}

 Using the density of such functions $m$ in $L^2(I, \mathbb{R})$, this implies that the equation defining the horizontal space of the elastic metric at $\gamma$ is
\begin{equation}\label{eqhorizontal}
\left(w'\cdot \operatorname{v}\right)' =  \frac{b}{a} \kappa \left( w'\cdot \operatorname{n}\right).
\end{equation}
\end{proof}

\vspace*{-8pt}
\subsection{Quotient elastic metrics}\label{sectionprojection}
Since the reparameterization group preserves the elastic metric $G^{a,b}$ defined above, it defines a quotient elastic metric on the quotient space $\Cb_1([0, 1])/\mathcal{G}([0, 1])$, which we will denote by $\overline{G}^{a,b}$. By Theorem~\ref{diffeoquotient}, this quotient space is identified with the submanifold $\mathcal{A}_1([0, 1])$, and we can pull back the quotient elastic metric $\overline{G}^{a,b}$ on $\mathcal{A}_1([0, 1])$.
We will denote the corresponding metric on $\mathcal{A}_1([0, 1])$ by $\widetilde{G}^{a,b}$. The value of the metric $\widetilde{G}^{a,b}$  on a tangent vector $w\in T_{\gamma}\mathcal{A}_1([0, 1])$ is the value of $\overline{G}^{a,b}([w], [w])$, where $[w]$ denotes the equivalence class of $w$ in the quotient space $T_{\gamma}\Cb_1([0, 1])/T_{\gamma}\mathcal{O}$. By definition of the quotient metric,
$$
\overline{G}^{a,b}([w], [w])= \inf_{u \in T_{\gamma}\mathcal{O}} G^{a,b}(w+u, w+u)
$$
where $u$ ranges over all tangent vectors in $ T_{\gamma}\mathcal{O}$. If $T_\gamma\Cb_1([0, 1])$ decomposes as $T_\gamma\Cb_1([0, 1]) = T_\gamma\mathcal{O}\oplus Hor_{\gamma}$, this minimum is achieved by the unique vector $P_h(w)\in [w]$ belonging to the horizontal space $ \textrm{Hor}_{\gamma}$ at $\gamma$.
In this case:
\begin{equation}\label{quotient_metric}
\widetilde{G}^{a,b}(w,w) = G^{a,b}(P_{h}(w), P_{h}(w)),
\end{equation}
where $P_{h}(w)\in T_{\gamma}\Cb_1([0, 1])$ is the projection of $w$ onto the horizontal space,  i.e., is the unique  \textit{horizontal} vector such that $w= P_{h}(w) + u$ with $u\in T_{\gamma}\mathcal{O}$.

\begin{proposition}
Let $w$ be a tangent vector to the manifold $\mathcal{A}_1([0, 1])$ at $\gamma$ and write $w' = \Phi \operatorname{n},$ where $\Phi$ is a real function in $\mathcal{C}^{\infty}([0, 1], \mathbb{R})$.
Then the projection $P_h(w)$ of $w\in T_{\gamma}\mathcal{A}_1([0, 1])$ onto the horizontal space $ \textrm{Hor}_{\gamma} $ reads $P_h(w) = w -m\operatorname{v}$ where $m\in \mathcal{C}^{\infty}([0, 1], \mathbb{R})$ is the unique solution of
\begin{equation}\label{horizontal_m}
-\frac{a}{b} m'' +  \kappa^2 m =  \kappa \Phi, \qquad m(0)=0, \quad m(1)=0.
\end{equation}
\end{proposition}

 \begin{proof}
Recall that a tangent vector $w$ to the manifold $\mathcal{A}_1([0, 1])$ at $\gamma$ satisfies $ w'\cdot \operatorname{v} = 0$, where $\operatorname{v}$ is the unit tangent vector field of the curve $\gamma$. Hence, for any $w \in T_{\gamma}\mathcal{A}_1([0, 1])$, the derivative $w'$ of $w$ with respect to the arc-length parameter reads
$
w' = \Phi \operatorname{n},
$
where $\Phi$ is a real function in $\mathcal{C}^{\infty}([0, 1], \mathbb{R})$.
One has
\begin{equation}\label{phm}
P_h(w)' = \Phi \operatorname{n} - m' \operatorname{v} - m \kappa \operatorname{n},
\end{equation}
hence $P_h(w)'\cdot \operatorname{v} = -m'$ and $P_h(w)'\cdot \operatorname{n} = (\Phi - m\kappa).$
The condition \eqref{eqhorizontal} for $P_h(w)$ to be horizontal is therefore \eqref{horizontal_m}.
Equation~\eqref{horizontal_m} is a particular case of  Sturm-Liouville equation $-(pm')' + q m = f$ with homogeneous boundary condition $m(0) = 0$ and $m(1) = 0$. Here $p = \frac{a}{b}>0$ and $q = \kappa^2\geq 0$.
The fact that equation \eqref{horizontal_m} has a unique solution follows from Lax-Milgram Theorem (see section 8.4 in \cite{Brezis}).
\end{proof}

For closed curves,  the tangent space to $\mathcal{A}_1^f(\mathbb{S}^1)$ at $\gamma$ contains the vector space of vector fields of the form $c\operatorname{v}$ where $c$ is a constant and $\operatorname{v} = \gamma'$. These vector fields generate the translation of base point, which is the natural action of the subgroup $\mathbb{S}^1$ of $\mathcal{G}(\mathbb{S}^1)$. One has  $$T_\gamma\mathcal{A}_1^f(\mathbb{S}^1)\cap T_\gamma\mathcal{O} = T_{\gamma}\left(\mathbb{S}^1\cdot\gamma\right),$$
where $\mathbb{S}^1\cdot \gamma =  \{s\mapsto \gamma(s+\tau), \tau \in \mathbb{S}^1\}$.
Therefore one can consider the horizontal projection $P_h\colon T_{[\gamma]}\mathcal{A}_1^f(\mathbb{S}^1)/\mathbb{S}^1 \rightarrow  \textrm{Hor}_{\gamma}$, where $[\gamma]$ denotes the projection of $\gamma$ on the quotient space $\mathcal{A}_1(\mathbb{S}^1)/\mathbb{S}^1$. We will denote by $[w]$ the projection of $w\in T_\gamma\mathcal{A}_1(\mathbb{S}^1)$ on the tangent space $T_{[\gamma]}\mathcal{A}_1(\mathbb{S}^1)/\mathbb{S}^1$.  Note that $[w] =\{ w + c\operatorname{v}, c\in\mathbb{R}\}$ and that $\int_0^1 w'(s) ds = 0$.

\begin{proposition}
Let $w$ be a tangent vector to the manifold $\mathcal{A}_1(\mathbb{S}^1)$ at $\gamma$ and write $w' = \Phi \operatorname{n},$ where $\Phi$ is a real function in $\mathcal{C}^{\infty}(\mathbb{S}^1, \mathbb{R})$ such that $\int_0^1\Phi(s)\operatorname{n}(s)\,ds = 0$.
Then the horizontal projection  $P_h([w])$ of $[w]$ onto the horizontal space reads $P_h([w]) = [w -m\operatorname{v}]$ where $m\in \mathcal{C}^{\infty}(\mathbb{S}^1, \mathbb{R})$ is the unique periodic solution of
\begin{equation}\label{horizontal_mclosed}
-\frac{a}{b} m'' +  \kappa^2 m =  \kappa \Phi.
\end{equation}
\end{proposition}

 \begin{proof}
As before the condition  for $w-m\operatorname{v}$ to be horizontal is \eqref{horizontal_m}.
The question is whether there exists a periodic solution $m$ of the equation
for given periodic functions $\kappa(x)$ and $\Phi(x)$. Since $\kappa(s+1)=\kappa(s)$ and $\Phi(s+1)=\Phi(s)$, we would like to satisfy $m'(1)=m'(0)$ and $m(1)=m(0)$. By the equation satisfied by $m$, it will imply that $m$ is a smooth periodic function on $\mathbb{S}^1$.
Let $y_1(s)$ and $y_2(s)$ be solutions of the equation $y''(s) - \kappa(s)^2 y(s)=0$, with initial conditions $y_1(0)=1$, $y_1'(0)=0$, $y_2(0)=0$, and $y_2'(0)=1$. Then Abel's formula implies that the Wronskian is
$$ W(s) = y_1(s)y_2'(s) - y_2(s)y_1'(s) \equiv 1,$$
and variation of parameters gives us the solution
$$ m(x) = c_1 y_1(s) + c_2 y_2(s) - y_1(s) \int_0^s \kappa(x)\Phi(x) y_2(x) \,dx + y_2(s) \int_0^s \kappa(x)\Phi(x) y_1(x) \, dx,$$
where $c_1=m(0)$ and $c_2=m'(0)$.

The question is how to choose $c_1$ and $c_2$ so that $m(1)=c_1$ and $m'(1)=c_2$. We clearly end up with the system
\begin{align*}
c_1 \big[y_1(1)-1\big] + c_2 y_2(1) &= By_1(1)-Ay_2(1) \\
c_2 y_1'(1) + c_2 \big[ y_2'(1)-1\big] &= By_1'(1)-Ay_2'(1),
\end{align*}
where $$A=\int_0^1 \kappa(x)\Phi(x)y_1(x)\,dx \qquad\text{and}\qquad B = \int_0^1 \kappa(x)\Phi(x)y_2(x)\,dx.$$
This has a solution if and only if the determinant
$$ \delta = [y_1(1)-1] [y_2'(1)-1] - y_2(1)y_1'(1)$$
is nonzero. Note that since the Wronskian is constant, we can write $\delta = 2 - y_2'(1) - y_1(1)$.

We now use reduction of order to write $y_2(s) = \phi(s) y_1(s)$, where
$$ \phi(s) = \int_0^s \frac{dx}{y_1(x)^2}.$$
It is obvious from the initial condition and the fact that $\kappa(s)^2$ is positive that $y_1(s)$ is strictly increasing for $s>0$, and $y_1'(s)$ is nonnegative for $s\ge 0$. Thus $\phi$ is always well-defined. We now have $y_2'(1) = \phi'(1) y_1(1) + \phi(1) y_1'(1)$, and thus our formula is
\begin{align*}
\delta &= 2 - \frac{1}{y_1(1)} - y_1'(1) \int_0^1 \frac{dx}{y_1(x)^2} - y_1(1) \\
&= -[y_1(1)-1/y_1(1)]^2 - y_1'(1) \int_0^1 \frac{dx}{y_1(x)^2}.
\end{align*}
We see that the only way this can be zero is if $y_1'(1)=0$ and $y_1(1)=y_1(1)$, and both these conditions are equivalent to $y_1(s)$ actually being constant, which only happens if $\kappa(s)$ is identically equal to zero on $[0,1]$. Hence unless the curve is a straight line, one can always solve the differential equation and get a unique periodic solution $m$.  Since $\gamma$ is a closed curve, $\gamma$ cannot be a straight line.
\end{proof}

Denote by $\mathbb{G}$ the Green function associated to equation~\eqref{horizontal_m}. By definition, the solution of
\begin{equation}\label{Green}
-\frac{a}{b} m'' +  \kappa^2 m =  \varphi,
\end{equation}
where $\varphi$ is any right-hand side, is
\begin{equation*}
m(s) = \int_0^1 \mathbb{G}(s, x)\varphi(x) dx,
\end{equation*}
where $m$ satisfies the additional condition:
\begin{itemize}
\item $m(0)=0$ and $m(1)=0$ for open curves,
\item $m$ is periodic for closed curves.
\end{itemize}

\begin{remark}{\rm
Using \eqref{phm}, observe that for any tangent vector $w\in T_{\gamma}\mathcal{A}_1(I)$ with $w' = \Phi \operatorname{n}$, one has
 \begin{equation}\label{quotient_metric22}
\begin{array}{cl}
\widetilde{G}^{a,b}(w,w )
 &= \int_0^1 \left(a (m')^2 + b (\Phi - m \kappa)^2\right) ds,
\end{array}
\end{equation}
where $m$ satisfies \eqref{horizontal_m} for open curves and \eqref{horizontal_mclosed} for closed curves.
}\end{remark}

We will also need the following expression of the quotient elastic metric on $\mathcal{A}_1([0, 1])$.

\begin{theorem}
Let $w$ and $z$ be two tangent vectors in $ T_{\gamma}\mathcal{A}_1([0, 1])$ with $w' = \Phi \operatorname{n} $ and  $z' = \Psi \operatorname{n}$, where $\Phi, \Psi \in \mathcal{C}^{\infty}([0, 1], \mathbb{R})$. Write $P_h(z) = z - p \operatorname{v}$, where $p$ satisfies $-a p'' + b \kappa^2 p =  b\kappa \Psi$ with $p(0)=p(1)=0$. Then the scalar product of $w$ and $z$ with respect to the quotient elastic metric $\widetilde{G}^{a,b}$ on the space of arc-length parameterized curves $\mathcal{A}_1([0, 1])$ reads
 \begin{equation}\label{hk}
\begin{array}{cl}
\widetilde{G}^{a,b}(w,z)
 &= \int_0^1 b  \Phi \left(\Psi-  \kappa p\right) ds.
\end{array}
\end{equation}
\end{theorem}

\begin{proof}
Denote respectively by $P_h(w)$ and $P_h(z)$ the projections of $w$ and $z$ on the horizontal space, and define $m, p \in \mathcal{C}^{\infty}([0, 1], \mathbb{R})$ by $P_h(w) = w - m \operatorname{v}$ and $P_h(w) = z- p \operatorname{v}$.
Since the horizontal space is the orthogonal space to $T_{\gamma}\mathcal{O}$ for the elastic metric $G^{a, b}$, one has
$$G^{a,b}(w, z) = G^{a,b}(P_h(w)-m \operatorname{v}, P_h(z)-p \operatorname{v}) = G^{a,b}(P_{h}(w), P_{h}(z)) + G^{a,b}(m \operatorname{v}, p \operatorname{v}). $$
It follows that
\begin{equation*}
\begin{array}{cl}
\widetilde{G}^{a,b}(w,z) &= G^{a,b}(P_h(w), P_h(z)) = G^{a,b}(w, z) - G^{a,b}(m \operatorname{v}, p \operatorname{v})\\
 &= \int_0^1 \left( b \Phi \Psi - a m' p' - b \kappa^2 m p\right) ds.
 \end{array}
\end{equation*}
After integrating the second term by parts, one has
 \begin{equation*}\label{quotient_metric2}
\begin{array}{cl}
\widetilde{G}^{a,b}(w, z) &
  = \int_0^1\left( b \Phi \Psi +  p (a m'' - b \kappa^2 m)\right) ds.
  \end{array}
\end{equation*}
Using the differential equation \eqref{horizontal_m} satisfied by the function $m$, we obtain \eqref{hk}.
\end{proof}

For closed curves, the same construction gives a Riemannian metric on the quotient space $\mathcal{A}_1^f(\mathbb{S}^1)/\mathbb{S}^1$. We can extend the definition of this metric to the space
$\mathcal{A}_1(\mathbb{S}^1)/\mathbb{S}^1$ by the same formula.
We get the following result:
\begin{theorem}
Let $w$ and $z$ be two tangent vectors in $ T_{\gamma}\mathcal{A}_1(\mathbb{S}^1)$ with $w' = \Phi \operatorname{n} $ and  $z' = \Psi \operatorname{n}$, where $\Phi, \Psi \in \mathcal{C}^{\infty}(\mathbb{S}^1, \mathbb{R})$. Write $P_h([z]) = [z - p \operatorname{v}]$, where $p$ satisfies $-a p'' + b \kappa^2 p =  b\kappa \Psi$ with periodic boundary conditions. Then the scalar product of $[w]$ and $[z]$ with respect to the quotient elastic metric $\widetilde{G}^{a,b}$ on the space of arc-length parameterized curves $\mathcal{A}_1(\mathbb{S}^1)/\mathbb{S}^1$ reads
 \begin{equation}\label{hkclosed}
\begin{array}{cl}
\widetilde{G}^{a,b}([w],[z])
 &= \int_0^1 b  \Phi \left(\Psi-  \kappa p\right) ds.
\end{array}
\end{equation}
\end{theorem}

\begin{proof}
Let us check that the expression of $\widetilde{G}^{a,b}([w],[z])$ does not depend on the representative of $[w]$ and $[z]$ chosen. Set $z_2 = z+ c\operatorname{v}$ for some constant $c\in\mathbb{R}$. Then $z_2' = z' + c\kappa\operatorname{n} = (\Psi + c\kappa)\operatorname{n}$. Denote by $p_2$ the solution of $-a p_2'' + b \kappa^2 p_2 =  b\kappa (\Psi+c\kappa)$ with periodic boundary conditions. Then $-a (p_2-c)'' + b \kappa^2 (p_2-c) =  b\kappa \Psi$.
By uniqueness of the solution of equation $-a p'' + b \kappa^2 p =  b\kappa \Psi$, one has $p = p_2 - c$. Therefore
$$
 \int_0^1 b  \Phi \left((\Psi+ c)-  \kappa p_2\right) ds =  \int_0^1 b  \Phi \left(\Psi-  \kappa p\right) ds.
$$
By symmetry, one also has the independence with respect to the representative of $[w]$.
\end{proof}

For closed curves, this Riemannian metric can be lifted in a unique way to a degenerate metric on $\mathcal{A}_1(\mathbb{S}^1)$ with only degeneracy along the fibers of the projection
$\mathcal{A}_1(\mathbb{S}^1)\rightarrow \mathcal{A}_1(\mathbb{S}^1)/\mathbb{S}^1$. The advantage of the degenerate lift is that it allows us to compare closed curves irrespective of the position of the base point.
This situation is analogous to the one encountered in Section \ref{modtranslation}, where the degeneracy of the metric was along the orbits by space translations. See also
\cite{PAMI} where this idea is used in the context of $2$-dimensional shapes.

\subsection{Definition and derivative of the energy functional}
In this section we will determine the gradient of the energy functional corresponding to the metric $\widetilde{G}^{a,b}$ on the spaces $\mathcal{A}_1([0, 1])$ and $\mathcal{A}_1(\mathbb{S}^1)/\mathbb{S}^1$ of arc-length parameterized curves. We will use the following conventions:
\begin{itemize}
\item[-] the arc-length parameter of curves in $\mathcal{A}_1(I)$ will be denoted by $s\in I$,
\item[-] the time parameter of a path in $\mathcal{A}_1(I)$ will be denoted by $t\in[0, T]$,
\item[-] the parameter $\varepsilon\in(-\delta, +\delta)$ will be the parameter of deformation of a path in $\mathcal{A}_1(I)$.
\end{itemize}
Consider a variation $\gamma\colon (-\delta, +\delta)\times[0, T]\times I\rightarrow \mathbb{R}^2$ of a smooth path in $\mathcal{A}_1(I)$.
In general in the following sections we will denote partial derivatives by subscripted index notations.
Note that, since any curve in $\mathcal{A}_1(I)$ is parameterized by arc-length, the arc-length derivative $\gamma_s$ of $\gamma$ is a unit vector in the plane for any values of the parameters $(\varepsilon, t, s)$, previously denoted by $\operatorname{v}$. For this reason, we will write it as
\begin{equation}\label{eta_s}
\gamma_s(\varepsilon, t, s) = \left(\cos \theta(\varepsilon, t, s), \sin \theta(\varepsilon, t, s)\right),
\end{equation}
where $\theta(\varepsilon, t, s)$ denotes a smooth lift of the angle between the $x$-axis and the unit vector $\operatorname{v}(\varepsilon, t, s) = \gamma_s(\varepsilon, t, s).$
In particular for closed curves, $\theta(\cdot, \cdot, 0) = 2\pi R + \theta(\cdot, \cdot, 1)$ where $R$ is the rotation number of the curve.

\begin{definition}
For any $\varepsilon\in (-\delta, +\delta)$, the function $t\mapsto \gamma(\varepsilon, t, \cdot)$ is a path in $\mathcal{A}_1(I)$, whose
 energy is defined as
$$
E(\varepsilon) = \frac{1}{2}\int_0^T \widetilde{G}^{a,b}(\gamma_t, \gamma_t) dt,
$$
where $\gamma_t$ is the tangent vector to the path $t\mapsto \gamma(\varepsilon, t, \cdot)\in \mathcal{A}_1(I)$.
\end{definition}

\begin{theorem}\label{energy_smooth}
Consider a variation $\gamma\colon (-\delta, +\delta)\times[0, T]\times I\rightarrow \mathbb{R}^2$ of a smooth path in $\mathcal{A}_1(I)$, with $\gamma_s(\varepsilon, t, s) = \left(\cos \theta(\varepsilon, t, s), \sin \theta(\varepsilon, t, s)\right)$ for some angle $\theta(\varepsilon, t, s)$. Then the energy as a function of $\varepsilon$ is given by
\begin{equation}
E(\varepsilon) = \frac{1}{2}\int_0^T \int_0^1 \left( a m_s^2 + b (\theta_t - \theta_s m)^2\right) \, ds\,dt,
\end{equation}
where $m$ is uniquely determined by the condition
\begin{equation}\label{horizontal_present}
-a m_{ss} + b \theta_s^2 m = b\theta_s \theta_t,
\end{equation}
with $m(0)=m(1)=0$ for $I = [0, 1]$ and periodic boundary conditions for $I = \mathbb{S}^1=\mathbb{R}/\mathbb{Z}$.
The derivative of the energy functional is given by
\begin{equation}\label{pdt_scal}
\frac{dE}{d\varepsilon}(0) =   \int_0^T \int_0^1 \theta_\varepsilon(t,s) \xi(t,s) \,ds \,dt,
\end{equation}
where
\begin{equation}\label{xi}
\frac{1}{b}\xi = -\theta_{tt} + \partial_t(\theta_s m) + \partial_s(\theta_t m) - \partial_s(\theta_s m^2).
\end{equation}
\end{theorem}

\begin{proof}
Equation~\eqref{eta_s} implies in particular that
\begin{equation*}
\gamma_{ss}(\varepsilon, t, s) = \theta_s(\varepsilon, t, s)\left(-\sin \theta(\varepsilon, t, s), \cos \theta(\varepsilon, t, s)\right) = \theta_s(\varepsilon, t, s)\operatorname{n}(\varepsilon, t, s),
\end{equation*}
where $s\mapsto \operatorname{n}(\varepsilon, t, s)= \left(-\sin \theta(\varepsilon, t, s), \cos \theta(\varepsilon, t, s)\right)$ is the normal vector field $\operatorname{n}$ along the parameterized curve $s\mapsto \gamma(\varepsilon, t, s)$.
In particular,  the curvature $\kappa(\varepsilon, t, s)$ of the curve $s\mapsto \gamma(\varepsilon, t, s)$ at $\gamma(\varepsilon, t, s)$ reads
\begin{equation*}
\kappa(\varepsilon, t, s) = \theta_s(\varepsilon, t, s).
\end{equation*}
For closed curves, one has $\theta_s(\varepsilon, t, s) = \theta_s(\varepsilon, t, s+1)$ since the curvature is a feature of the curve.
Furthermore the arc-length derivative of the tangent vector $\gamma_t$ along the path $t\mapsto \gamma(\varepsilon, t, s)$ reads
\begin{equation*}
\gamma_{ts}(\varepsilon, t, s) = \gamma_{st}(\varepsilon, t, s) = \theta_t(\varepsilon, t, s)\operatorname{n}(\varepsilon, t, s).
\end{equation*}
For $I = \mathbb{S}^1$, since $\gamma$ is a path of closed curves, $\gamma_t(\varepsilon, t, s) = \gamma_t(\varepsilon, t, s+1)$ and $\theta_t(\varepsilon, t, s) = \theta_t(\varepsilon, t, s+1)$.
Denote by $m\in\mathcal{C}^{\infty}([0,T]\times I, \mathbb{R})$ the solution, for each fixed $t$, of
\begin{equation}\label{mequation}
-\frac{a}{b} m_{ss}(t,s) + \theta_s^2(t,s) m(t,s) = \theta_s(t,s) \theta_t(t,s),
\end{equation}
with $m(t,0)=m(t,1)=0$ for $I = [0, 1]$ and periodic boundary conditions for $I = \mathbb{S}^1$, i.e.,
\begin{equation}\label{m}
m(t,s) = \int_0^1 \mathbb{G}(t;s, x)  \theta_x(t,x)\theta_t(t,x) dx,
\end{equation}
where $\mathbb{G}$ is the (time-dependent) Green function associated to equation~\eqref{Green} (we have omitted the dependency on $\varepsilon$ here in order to improve readability).
Using the expression of the metric $\widetilde{G}^{a,b}$ given in \eqref{quotient_metric22} with $\Phi = \theta_t$ and $\kappa = \theta_s$, one has
$$
E(\varepsilon) = \frac{1}{2}\int_0^T \int_0^1 \left( a m_s^2 + b (\theta_t - \theta_s m)^2\right) \, ds\,dt.
$$

Note that the $\varepsilon$-derivative $\gamma_{\varepsilon}$ at $\varepsilon = 0$ is a vector field along the path $t\mapsto \gamma(0, t, s)$. Hence for any fixed parameter $t\in[0, T]$,
$s\mapsto \gamma_{\varepsilon}(0, t, s)$ is an element of the tangent space $T_{\gamma(0, t, \cdot)}\mathcal{A}_1(I)$ whose arc-length derivative reads
\begin{equation}
\gamma_{\varepsilon s}(0, t, s) = \theta_\varepsilon(0, t, s)\operatorname{n}(0, t, s).
\end{equation}
The derivative of the energy functional with respect to the parameter $\varepsilon$ is therefore
$$
\frac{dE}{d\varepsilon}(0) =  \int_0^T \int_0^1 am_sm_{s\varepsilon} + b(\theta_t - \theta_s m)(\theta_{t\varepsilon} - \theta_{s\varepsilon} m - \theta_s m_{\varepsilon}) \, ds\,dt.
$$
Integrate the first term by parts in $s$, and we obtain
\begin{equation*}
\begin{split}
\frac{dE}{d\varepsilon}(0) &= \int_0^T \int_0^1 b(\theta_t - \theta_s m) (\theta_{t\varepsilon} - m \theta_{s\varepsilon}) \, ds\,dt\\
&+ \int_0^T \int_0^1 m_{\varepsilon} (-am_{ss} - b\theta_t\theta_s  + b \theta_s^2 m) \, ds\,dt,
\end{split}
\end{equation*}
and the last term vanishes by equation \eqref{horizontal_present}.
Integrating by parts in $s$ and $t$ to isolate $\theta_{\varepsilon}$, we obtain
\eqref{pdt_scal}--\eqref{xi}
\end{proof}

\subsection{Gradient of the energy functional}\label{gradient_section}
In Theorem~\ref{energy_smooth}, the derivative of the energy functional is expressed as the integral of an $L^2$-product, i.e., as a $1$-form. In order to obtain the gradient of the energy functional, we need to find the vector corresponding to  this $1$-form via the quotient elastic metric $\widetilde{G}^{a,b}$ on $\mathcal{A}_1(I)$. In other words, the aim is to rewrite the derivative of the energy functional, given by \eqref{pdt_scal}--\eqref{xi}, as
\begin{equation}\label{energygradient}
\frac{dE}{d\varepsilon}(0) = \int_{0}^{T} \widetilde{G}^{a,b}({\gamma}_\varepsilon,\grad E(\gamma)) dt,
\end{equation}
for some vector field $\grad E(\gamma)$ along the path $\gamma$ in $\mathcal{A}_1(I)$. Deforming the path $\gamma$ in the opposite direction of $\grad E(\gamma)$ will then give us an efficient way to minimise the path-energy of $\gamma$, and a path-straightening algorithm will allow us to find approximations of geodesics.

Based on equations \eqref{hk} and \eqref{hkclosed}, finding this Riemannian gradient now reduces to solving the following problem for each fixed time: given functions $\kappa(s)$ and $\xi(s)$, find a function $\beta(s)$ such that
\begin{equation}\label{betahard}
\beta(s) - \kappa(s) h(s) = \xi(s), \qquad \text{where } ah''(s) - b\kappa(s)^2 h(s) = -b\kappa(s) \beta(s),
\end{equation}
with boundary conditions $h(0)=h(1)=0$ for open curves, $h(0) = h(1)$ and $h'(0) = h'(1)$ for closed curves.
At first glance this problem seems rather tricky, since in terms of the Green
function $\mathbb{G}$ defined by \eqref{Green}, we have $h=\mathbb{G}\star (\kappa \beta)$, and so \eqref{betahard} appears to become
$h-\kappa \mathbb{G} \star (\kappa h) = \xi$, which would require inverting the operator $I-M_{\kappa} K M_{\kappa}$, where $K$ is the operator $h\mapsto \mathbb{G}\star h$ and $M_{\kappa}$ is the operator of multiplication by $\kappa$. What is remarkable in the following theorem is that
this computation actually ends up being a lot simpler than expected due to some nice cancellations.

\begin{theorem}\label{magnitude2}
Consider a variation $\gamma\colon (-\delta, +\delta)\times[0, T]\times I\rightarrow \mathbb{R}^2$ of a smooth path in $\mathcal{A}_1(I)$, with $\gamma_s(\varepsilon, t, s) = \left(\cos \theta(\varepsilon, t, s), \sin \theta(\varepsilon, t, s)\right)$ for some angle $\theta(\varepsilon, t, s)$.
Then the gradient $\grad E$ determined by formula \eqref{energygradient} satisfies
$(\grad E)_s(0, t, s) = \beta(t, s) \operatorname{n}(t ,s)$ with
\begin{equation}\label{gradient_energy_smoothbeta}
\begin{split}
\beta(0, t, s) = & \frac{1}{b} \xi(0, t, s) - \frac{1}{a} \theta_s(0, t, s) \int_{0}^s\left(\int_{0}^x \theta_s(0, t, y)\xi(0, t, y) dy\right) dx \\&+\frac{1}{a} \kappa(s) s \int_0^1\left(\int_0^x \kappa(y)\xi(y) dy\right) dx,
\end{split}
\end{equation}
or equivalently
\begin{multline}\label{gradient_energy_smooth}
\beta(t, s) = \frac{1}{b}\xi(t, s) - \theta_s(t,s) m_t(t,s) - \tfrac{b}{2a} C(t) s\theta_s(t,s) \\
 + \tfrac{1}{2} \theta_s(t,s) \int_0^s \left(m_x(t,x)^2 + \tfrac{b}{a} \theta_x(t,x)^2 m(t,x)^2 - \tfrac{b}{a} \theta_t(t,x)^2\right)\,dx,
\end{multline}
where $\xi$ is given by \eqref{xi}, $m$ satisfies \eqref{mequation}, and  $C(t)$ is given by
\begin{equation}\label{Cdef}
C(t) = \int_0^1 \theta_s(t,s) \theta_t(t,s) m(t,s) \, ds - \int_0^1 \theta_t(t,s)^2 \, ds.
\end{equation}
\end{theorem}

\begin{proof}
By Theorem~\ref{energy_smooth}, the derivative of the energy functional is the integral of $\langle \theta_\varepsilon, \xi \rangle$ where $\xi$ is given by \eqref{xi}. Recall that $\theta_\varepsilon$ is related to the derivative $\gamma_\varepsilon$ by $\gamma_{\varepsilon s} = \theta_\varepsilon \operatorname{n}$.
Comparing with the expression of the quotient elastic metric \eqref{hk}, it follows that
 \begin{equation*}
\begin{array}{cl}
 \langle \theta_{\varepsilon}, \xi\rangle & = \widetilde{G}^{a,b}(\gamma_{\varepsilon}, \grad E),
\end{array}
\end{equation*}
where $\xi = b \left(\beta-  \kappa h\right)$, and where $\beta$ and $h$ are related to $\grad E$ by $(\grad E)_s = \beta\operatorname{n}$ and $-a h'' + b \kappa^2 h =  b\kappa \beta$.
Note that $\xi$ determine the functions $\beta$ and $h$ since the relation
$b\beta = \xi + b\kappa h$ implies
$$
-a h'' = \kappa \xi.
$$
A first integration gives
$$
h'(x) = -\frac{1}{a}\int_{0}^x \kappa(y) \xi(y) dy + c_1,
$$
for some constants $c_1$ and a second integration gives
\begin{equation}\label{hsimple}
h(s) = -\frac{1}{a}\int_0^s\left(\int_{0}^x \kappa(y) \xi(y) dy\right) dx + c_1 s + c_2,
\end{equation}
for some other constant $c_2$.

For open curves, using the condition $h(0) = h(1) = 0$, we obtain $c_2 = 0$ and $$c_1 = \frac{1}{a}\int_0^1\left(\int_{0}^x \kappa(y) \xi(y) dy\right) dx.$$ Therefore
\begin{equation*}
h(s) = \frac{1}{a}\int_{0}^s\left(\int_{0}^x -\kappa(y)\xi(y) dy\right)~dx +\frac{1}{a} s \int_0^1\left(\int_{0}^x \kappa(y) \xi(y) dy\right) dx,
\end{equation*}
and
\begin{equation*}
\beta(s) = \frac{1}{b} \xi(s) - \frac{1}{a}  \kappa(s) \int_{0}^s\int_{0}^x \kappa(y)\xi(y) dy~dx+\frac{1}{a} \kappa(s) s \int_0^1\left(\int_{0}^x \kappa(y) \xi(y) dy\right) dx.
\end{equation*}
Substituting $\kappa = \theta_s$ gives \eqref{gradient_energy_smoothbeta}.

Moreover
by formula \eqref{xi} we have
 \begin{equation}\label{xistep1}
 \kappa \xi = -\theta_s \theta_{tt} + 2\theta_s \theta_{ts} m + \theta_s^2 m_t + \theta_t\theta_s m_s - \theta_{ss} \theta_s m^2 - 2\theta_s^2 m m_s.
 \end{equation}
 But also differentiating \eqref{mequation} in time gives
$$ am_{sst} - b\theta_s^2 m_t = 2b\theta_s\theta_{st} m - b\theta_{st} \theta_t - b\theta_s\theta_{tt},$$
and eliminating $\theta_s\theta_{tt}$ in \eqref{xistep1} gives the equation
$$ \kappa \xi = \tfrac{a}{b} m_{sst} + \theta_{st} \theta_t + \theta_s \theta_t m_s - \theta_{ss}\theta_s m^2 - 2\theta_s^2 mm_s.$$
Now substitute from \eqref{mequation} the relation $\theta_s\theta_t = \theta_s^2 m - \tfrac{a}{b} m_{ss}$, and we obtain
$$
 -\tfrac{a}{b} h_{ss} = \kappa \xi = \tfrac{a}{b} m_{sst} + \theta_{st}\theta_t - \tfrac{a}{b} m_s m_{ss} - \theta_{ss}\theta_s m^2 - \theta_s^2 mm_s.
$$
The right side is now easy to integrate in $s$, and we get
\begin{equation}\label{kappaxi}
 -ah_s = a m_{st} + \tfrac{1}{2} b \theta_t^2 - \tfrac{1}{2} a m_s^2 - \tfrac{1}{2} b \theta_s^2 m^2 + \tfrac{b}{2a} C,
  \end{equation}
where the constant $C$ is chosen so that both sides integrate to zero between $s=0$ and $s=1$ (since $h(0)=h(1)=0$). Multiplying both sides of
\eqref{mequation} by $m$ and integrating from $s=0$ to $s=1$, we conclude that $C(t)$ satisfies \eqref{Cdef}.
Another integration in $s$ gives the formula
\begin{equation}\label{heq}
\begin{split}
h(t,s) = &-m_t(t,s) + \tfrac{1}{2} \int_0^s m_x(t,x)^2 \, dx + \tfrac{b}{2a} \int_0^s \theta_x(t,x)^2m(t,x)^2 \\& - \theta_t(t,x)^2 \, dx - \tfrac{b}{2a} C(t)s.
\end{split}
\end{equation}
Since $m(t,0)=m(t,1)=0$ for all $t$, this clearly vanishes at $s=0$ as it should; furthermore it is easy to check that it also vanishes at $s=1$ by definition of $C$. Plugging $h$ given by \eqref{heq} into the formula $\beta = \frac{1}{b}\xi + \kappa h$, we obtain
\eqref{gradient_energy_smooth} as desired.

For closed curves,  using the conditions $h(0) = h(1)$ and $h'(0) = h'(1)$ in \eqref{hsimple}, we obtain $c_1 = \frac{1}{a}\int_0^1\left(\int_{0}^x \kappa(y) \xi(y) dy\right) dx$ and the condition $\int_0^1 \kappa(s) \xi(s)\,ds = 0$,
which is satisfied by \eqref{kappaxi} since the right hand side is periodic. Note that there is no condition on $c_2$ as expected. We take $c_2 = 0$ in order to match the formula for open curves.
\end{proof}

\begin{remark}
Given the derivative of the gradient flow $(\grad E)_s(0, t, s) = \beta(t, s) \operatorname{n}(t ,s)$ with $\beta(t, s)$ given by \eqref{gradient_energy_smoothbeta} or \eqref{gradient_energy_smooth}, we have flexibility in the choice of the constant of integration to obtain $\grad E$. This is related to the fact that the curves are considered modulo translations (see Section~\ref{modtranslation}). In the numerics we used the condition $\grad E(0) = 0$, which corresponds to representing curves modulo translations as curves starting at the origin. Furthermore, there is no guarantee that $\int_0^1 \beta(t, s) \operatorname{n}(t ,s) =0$, in other words the gradient may not preserve the closedness condition. Since the space of closed curves is a codimension $2$ submanifold of the vector space of open curves, we have to project the gradient of the energy functional to the tangent space of the space of closed curves. This projection is given by $\grad E(s) \mapsto \grad E(s) - s\int_0^1\grad E(x) dx.$
\end{remark}

\section{Quotient elastic metrics on arc-length parameterized piecewise linear curves}\label{discretization}

\subsection{Notation}
Let us consider a ``chain'' given by points joined by rigid rods of length $1/n$.
We denote the points by $\gamma_k$ for $1\le k\le n$, and periodicity is enforced by requiring $\gamma_{n+1} = \gamma_1$
and $\gamma_0 = \gamma_n$. We let $$\textrm{v}_k = n(\gamma_{k+1}-\gamma_k)$$ denote the unit vectors along the rods,
and $\theta_k$ be the angle between the $x$-axis and $\textrm{v}_k$, so that
$$
\textrm{v}_k = (\cos \theta_k, \sin\theta_k).
$$
The unit normal vectors are defined by
$$
\textrm{n}_k = (-\sin\theta_k, \cos \theta_k).
$$
We will also introduce the variation of the angles $\theta_k$:
$$
\Delta_k = \theta_k -\theta_{k-1}.
$$
Vector fields along a chain are denoted by sequences $w = (w_k: 1\le k\le n)$. A vector field $w$ preserves the arc-length parameterization
if and only if
$$\frac{d}{dt}\big|_{t=0} \lvert \gamma_{k+1}(t)-\gamma_k(t)\rvert^2 = \tfrac{2}{n} \langle w_{k+1}-w_k, \textrm{v}_k\rangle = 0,$$
for any $k$, where $\gamma_k(t)$ is any variation of  $\gamma_k$ satisfying $w_k = \gamma_k'(0)$. In particular, any vector field preserving the arc-length parameterization satisfies
$$
w_{k+1}-w_k = \tfrac{1}{n} \phi_k  \textrm{n}_k,
$$
for some sequence $\phi = (\phi_k: 1\le k\le n)$.

\subsection{Discrete version of the elastic metrics}

The discrete elastic metric is given by
\begin{equation}\label{elastic_discrete}
G^{a,b}(w,w) = n\sum_{k=1}^n \big( a\langle w_{k+1}-w_k, \textrm{v}_k\rangle^2 + b\langle w_{k+1}-w_k, \textrm{n}_k\rangle^2\big),
\end{equation}
which clearly agrees with \eqref{elastic_metric_abc} in the limit as $n\to \infty$ using $w(k/n) = w_k$. In addition this metric
has the same property as \eqref{elastic_metric_abc} in that the $a$ term disappears when $w$ is a field that preserves the arc-length
parameterization.
For two vector fields $w$ and $z$, the expression of their $G^{a,b}$ scalar product reads
\begin{equation}\label{elastic_discrete_product}
\begin{split}
G^{a,b}(w,z) =  n\sum_{k=1}^n &\big( a\langle w_{k+1}-w_k, \textrm{v}_k\rangle\langle z_{k+1}-z_k, \textrm{v}_k\rangle \\&+ b\langle w_{k+1}-w_k, \textrm{n}_k\rangle\langle z_{k+1}-z_k, \textrm{n}_k\rangle\big).
\end{split}
\end{equation}
For further use note that if $w$ preserves the arc-length parameterization and $z$ is arbitrary,
\begin{equation}\label{elastic_discrete_product2}
G^{a,b}(w,z) = n\sum_{k=1}^n b\langle w_{k+1}-w_k, \textrm{n}_k\rangle\langle z_{k+1}-z_k, \textrm{n}_k\rangle.
\end{equation}

\subsection{Horizontal space for the discrete elastic metrics }
Assume now that $w$ preserves the arc-length parameterization, and write $n (w_{k+1} - w_k) =  \phi_k \textrm{n}_k$ for
some numbers $\phi_k$.
The ``vertical vectors'' will still be all those of the form $u_k = g_k \textrm{v}_k$ for some numbers $g_k$,
although it is not clear
in the discrete context if these actually represent the nullspace of a projection as in the smooth case. Let us show the following:

\begin{theorem}\label{horizontal_discrete_projection_thm}
If $(w_k : 1\le k\le n)$ satisfies $n(w_{k+1} - w_k) = \phi_k \operatorname{n}_k$, then its projection onto the orthogonal space to the space
spanned by vectors of the form $u_k = g_k \operatorname{v}_k$,
with respect to the discrete elastic metric \eqref{elastic_discrete} is
\begin{equation}\label{basic_projection}
P_h(w) = w_k - m_k \operatorname{v}_k
\end{equation}
where the numbers $m_k$ satisfy
\begin{equation}\label{vertical_discrete_projection}
\begin{split}
\tfrac{b}{n}  \sin{\Delta_k} \phi_{k-1} &=  (a + a\cos^2{\Delta_k} + b \sin^2{\Delta_k}) m_k \\&- a\cos{\Delta_k}m_{k-1} - a\cos{\Delta_{k+1}} m_{k+1}
\end{split}
\end{equation}
with $\operatorname{v}_k = (\cos{\theta_k}, \sin{\theta_k})$ and $\Delta_k = \theta_k-\theta_{k-1}$.
\end{theorem}

\begin{proof}
For every vertical vector $(g_k \textrm{v}_k)$ for any numbers $g_k$, we want to see that $G^{a,b}(w-m\textrm{v}, g\textrm{v}) = 0$. We therefore get
\begin{align*}
0 &= \sum_{k=1}^n  a \langle w_{k+1}-w_k - m_{k+1}\textrm{v}_{k+1} + m_k \textrm{v}_k, \textrm{v}_k\rangle \langle g_{k+1}\textrm{v}_{k+1}-g_k\textrm{v}_k, \textrm{v}_k\rangle \\
&\qquad\qquad + b\langle w_{k+1}-w_k - m_{k+1}\textrm{v}_{k+1} + m_k\textrm{v}_k, \textrm{n}_k\rangle \langle g_{k+1}\textrm{v}_{k+1} - g_k \textrm{v}_k, \textrm{n}_k \rangle  \\
&= \sum_{k=1}^n a ( m_k -m_{k+1} \langle \textrm{v}_{k+1},\textrm{v}_k\rangle) (g_{k+1}\langle \textrm{v}_{k+1},\textrm{v}_k\rangle - g_k) \\
&\qquad\qquad + b ( \tfrac{1}{n} \phi_k - m_{k+1} \langle \textrm{v}_{k+1}, \textrm{n}_k\rangle) g_{k+1} \langle \textrm{v}_{k+1}, \textrm{n}_k\rangle.
\end{align*}
Using the identities
$$
\langle \textrm{v}_{k+1}, \textrm{v}_{k}\rangle = \cos \theta_{k+1} \cos \theta_k + \sin \theta_{k+1} \sin \theta_k = \cos \Delta_{k+1},
$$
and
$$
\langle \textrm{v}_{k+1}, \textrm{n}_k\rangle = -\cos \theta_{k+1} \sin\theta_k + \sin \theta_{k+1} \cos \theta_k = \sin \Delta_{k+1},
$$
one gets
\begin{align*}
0 &= \sum_{k=1}^n g_k \Big[ a (m_{k-1} - m_k \cos{\Delta_k}) \cos{\Delta_k} - a(m_k - m_{k+1} \cos{\Delta_{k+1}}) \\
&\qquad\qquad + \tfrac{1}{n} b\phi_{k-1} \sin{\Delta_k}
- b m_k \sin^2{\Delta_k}\Big],
\end{align*}
after reindexing. Since this must be true for every choice of $g_k$, we obtain \eqref{vertical_discrete_projection}.
\end{proof}
\begin{remark}\label{stability}{\rm
It is easy to check that \eqref{vertical_discrete_projection} is a discretization of \eqref{horizontal_present}, as expected.
Note that equation~\eqref{vertical_discrete_projection} can be rewritten as
\begin{equation*}
\frac{b}{n}\left(\begin{smallmatrix}
 \sin \Delta_1 \phi_n \\ \sin \Delta_2 \phi_1 \\ \sin \Delta_3 \phi_2\\ \vdots \\ \sin \Delta_{n-1} \phi_{n-2} \\\sin \Delta_n \phi_{n-1}
\end{smallmatrix}\right)
= \operatorname{T}
\left(\begin{smallmatrix}
m_1 \\m_2  \\m_3 \\ \vdots \\m_{n-1} \\m_{n}
\end{smallmatrix}\right)
\end{equation*}
where $\operatorname{T}$ is a cyclic tridiagonal matrix of the form
\begin{equation}\label{T}
\operatorname{T} = \left(\begin{smallmatrix}
d_1& \tau_2 & 0 & 0 & \cdots & 0 & 0  &\tau_1 \\
\tau_2 & d_2 & \tau_3 & 0 &   \cdots & 0 &  0 & 0\\
0 & \tau_3 &  d_3 & \tau_4 & \cdots &0 &  0 & 0\\
\vdots & \vdots & \vdots & \vdots  & \cdots & \vdots & \vdots & \vdots  \\
0 & 0 & 0 & 0 &  \cdots &   \tau_{n-1} & d_{n-1} & \tau_n\\
\tau_1 & 0 & 0 & 0 &  \cdots &   0 & \tau_{n} & d_{n}  \\
\end{smallmatrix}\right)
\end{equation}
with $d_k =  a + a \cos^2 \Delta_k + b \sin^2 \Delta_k$ and $\tau_k = -a \cos\Delta_k$.
Note that $\operatorname{T}$ is a small deformation of a tridiagonal matrix which can be inverted in $O(n)$ operations using the Thomas algorithm. Observe that $d_k> \tau_k+\tau_{k+1}$ as soon as $\cos \Delta_{k+1}> -\frac{3}{4}$, hence the matrix $\operatorname{T}$ is strictly dominant as soon as the angles between two successive rods are small enough, and this can be easily achieved by raising the number of points. This implies that the Thomas algorithm is numerically stable (\cite{High}).
See \cite{Dubeau} where algorithms are presented to invert cyclic tridiagonal matrices.
Other  algorithms for the solution of cyclic tridiagonal systems are given for example in \cite{Clive}.
}
\end{remark}

\subsection{Definition and derivative of the energy functional in the discrete case}

Consider a path $t\mapsto \gamma_k(t)$, $0\leq t\leq T$,   preserving the arc-length parameterization (i.e., the length of the rods) and connecting two positions of the chain $\gamma_{1,k}$ and $\gamma_{2,k}$. Write $$\gamma_{k+1}(t) - \gamma_k(t) = \tfrac{1}{n}\textrm{v}_k(t) = \tfrac{1}{n}(\cos{\theta_k(t)}, \sin{\theta_k(t)}).$$ We will use a dot for the differentiation with respect to the parameter $t$ along the path. In particular $w = \dot{\gamma}$ is a vector field along the chain $\gamma$ satisfying
$$w_{k+1}(t) - w_k(t) =\tfrac{1}{n} \dot{\theta}_k(t) \textrm{n}_k(t).$$
Let $\Delta_k(t) = \theta_k(t) - \theta_{k-1}(t)$.
Given a  variation $\varepsilon\mapsto\gamma_k(\varepsilon, t)$, $\varepsilon \in (-\delta, \delta)$, of the path $\gamma_k(0, t) = \gamma_k(t)$
 preserving the arc-length parameterization, let us compute the energy functional for the discrete elastic metrics and its derivative at $\varepsilon = 0$. We will use a subscript $\varepsilon$ for the differentiation with respect to $\varepsilon$, in particular we will use the notation
$$
\frac{d}{d\varepsilon}|_{\varepsilon=0}\left(\gamma_{k+1}(\varepsilon, t) - \gamma_k(\varepsilon, t)\right) = \tfrac{1}{n} \theta_{\varepsilon, k}~ \textrm{n}_k(0,t).
$$

\begin{theorem}\label{energy_discrete_derivative}
Suppose we have a family of curves $\gamma_k(\varepsilon, t)$ depending on time and joining fixed curves $\gamma_{1,k}$ and $\gamma_{2,k}$
(which is to say that $\gamma_k(\varepsilon, 0) = \gamma_{1,k}$ and $\gamma_k(\varepsilon, T) = \gamma_{2,k}$ for all $\varepsilon$ and $k$).
Then the energy as a function of $\varepsilon$ is
\begin{equation}\label{energyepsilon}
E(\varepsilon) = \tfrac{n}{2} \int_0^T \sum_{k=1}^n \Big( a(m_k - m_{k+1} \cos{\Delta_{k+1}})^2 + b(\tfrac{1}{n} \dot{\theta}_k - m_{k+1} \sin{\Delta_{k+1}})^2\Big) \, dt,
\end{equation}
where $m$ satisfies \eqref{vertical_discrete_projection} with $\phi_k = \dot{\theta}_k$. Its derivative at $\varepsilon=0$ is given by
$$ \frac{dE}{d\varepsilon}(0) =  \int_0^T\frac{1}{n} \sum_{k=1}^n \theta_{\varepsilon, k}(0, t) \xi_k(t) \, dt,$$
where $\xi_k$ is given by
\begin{multline}\label{xidef}
\xi_k = -b \ddot{\theta}_k +  b n (\dot{m}_{k+1} \sin{\Delta_{k+1}}  + m_{k+1} \cos \Delta_{k+1} \dot{\theta}_{k+1} - m_k \cos \Delta_k \dot{\theta}_{k-1})\\
+ n^2 (b-a) (m_k^2 \sin{\Delta_k} \cos{\Delta_k} -  m_{k+1}^2 \sin{\Delta_{k+1}} \cos{\Delta_{k+1}}) \\
+ a n^2 m_k (m_{k-1} \sin{\Delta_k} - m_{k+1} \sin{\Delta_{k+1}}).
\end{multline}
\end{theorem}

\begin{proof}
By Theorem~\ref{horizontal_discrete_projection_thm},
 the horizontal projection of the velocity vector $w = \dot{\gamma}$ is given by
$P_h(w) =  w_k - m_k \textrm{v}_k$ where $m$ satisfies \eqref{vertical_discrete_projection} with $\phi_k = \dot{\theta}_k$. Hence the energy is
\begin{multline}
E(\varepsilon) = \tfrac{n}{2}  \int_0^T \sum_{k=1}^{n}a \langle w_{k+1}-w_k - m_{k+1}\textrm{v}_{k+1} + m_k \textrm{v}_k, \textrm{v}_k\rangle^2 \\
+ b \langle w_{k+1}-w_k - m_{k+1}\textrm{v}_{k+1} + m_k \textrm{v}_k, \textrm{n}_k\rangle^2 \, dt,
\end{multline}
which reduces to \eqref{energyepsilon}.

To compute the derivative of the energy functional, we first simplify \eqref{energyepsilon} by expanding and reindexing to obtain
\begin{multline*}
E(\varepsilon) = \tfrac{n}{2} \sum_{k=1}^n \int_0^T \Big( \tfrac{b}{n^2} \dot{\theta}_{k-1}^2 - 2\tfrac{b}{n} m_k \dot{\theta}_{k-1} \sin{\Delta_k} + b m_k^2 \sin^2{\Delta_k}
\\
+ a m_{k}^2 - 2 a m_{k-1} m_k \cos{\Delta_k} + a m_k^2 \cos^2{\Delta_k}\Big) \, dt.
\end{multline*}
Now let $\psi_k = \frac{\partial \theta_k}{\partial \varepsilon}\vert_{\varepsilon=0}$, $\nu_k = \psi_k - \psi_{k-1}$, and $g_k = \frac{\partial m_k}{\partial \varepsilon}\vert_{\varepsilon=0}$.
We then get
\begin{align*}
\frac{dE}{d\varepsilon}(0) &= n \sum_{k=1}^n \int_0^T \Big( \tfrac{b}{n^2} \dot{\theta}_{k-1} \dot{\psi}_{k-1} - \tfrac{b}{n} m_k \dot{\psi}_{k-1} \sin{\Delta_k} - \tfrac{b}{n} m_k \dot{\theta}_{k-1} \cos{\Delta_k} \nu_k \\
&\qquad\qquad + (b-a)m_k^2 \sin{\Delta_k}\cos{\Delta_k} \nu_k + am_{k-1}m_k \sin{\Delta_k} \nu_k
 \Big) \, dt \\
&\qquad\qquad + n \sum_{k=1}^n \int_0^T g_k \Big( -\frac{b}{n} \dot{\theta}_{k-1} \sin{\Delta_k} + b m_k \sin^2{\Delta_k} + a m_k \\
&\qquad\qquad -a m_{k-1} \cos{\Delta_k}- a m_{k+1} \cos{\Delta_{k+1}} + a m_k \cos^2{\Delta_k}\Big) \, dt.
\end{align*}
But notice that the term multiplied by $g_k$ vanishes since $m_k$ satisfies \eqref{vertical_discrete_projection}; hence it is not necessary to compute the variation $g_k$.
All that remains is to express every term in $\frac{dE}{d\varepsilon}(0)$ in terms of $\psi_k$ either by reindexing or integrating by parts in time, which is straightforward and leads to \eqref{xidef}.
\end{proof}

\subsection{Gradient of the discrete energy functional}
Let us compute the gradient of the discrete energy functional with respect to the quotient elastic metric $G^{a,b}$.
Considering equation~\eqref{xidef}, let us first compute $\dot{m}_k$.

\begin{lemma}\label{mdotlemma}
Let $\mathbb{G}$ denote the inverse matrix of  the matrix $\operatorname{T}$ in \eqref{T}, so that
\begin{equation}\label{m_j}
 m_j = \sum_{k=1}^n \mathbb{G}_{jk} \frac{b}{n} \phi_{k-1} \sin{\Delta_k} \qquad \text{for all $j$},
 \end{equation}
where $\Delta_k = \theta_k-\theta_{k-1}$ for some angles $\theta_k$. If $\theta_k(t)$ depends on time and its derivative is denoted by $\phi_k(t) = \dot{\theta}_k(t)$, then
we have the formula
\begin{multline}\label{mdotformula}
\dot{m_j} = \sum_{k=1}^n \mathbb{G}_{jk} \Big( \frac{b}{n}\sin{\Delta_k} \ddot{\theta}_{k-1} + \frac{b}{n} \cos{\Delta_k} \dot{\theta}_{k-1} \dot{\Delta}_k
+ 2(a-b) \sin{\Delta_k}\cos{\Delta_k} m_k \dot{\Delta}_k \\
- a \sin{\Delta_k} m_{k-1} \dot{\Delta}_k - a \sin{\Delta_{k+1}} m_{k+1} \dot{\Delta}_{k+1}\Big).
\end{multline}
\end{lemma}

\begin{proof}
We just compute the time derivative of each term of equation \eqref{vertical_discrete_projection} and notice that the terms involving $\dot{m}_k$ are
$$ b \sin^2{\Delta_k} \dot{m}_k + a\dot{m}_k + a\cos^2{\Delta_k} \dot{m}_k - a\cos{\Delta_k}\dot{m}_{k-1} - a\cos{\Delta_{k+1}} \dot{m}_{k+1}.$$
Hence we need to invert the same matrix $\operatorname{T}$ to solve for $\dot{m}_k$ as we do to solve for $m_k$. The remainder is straightforward.
\end{proof}

Finally let us rewrite the $l^2$-product in \eqref{xidef} as an $G^{a,b}$-inner product, analogously to Theorem~\ref{magnitude2}.
\begin{proposition}\label{gradient_discrete_prop}
Let $w$ and $z$ be two vector fields along $\gamma$  with $n(w_{k+1}-w_k) = \alpha_k \operatorname{n}_k$ and $n(z_{k+1}-z_k) = \beta_k \operatorname{n}_k$ for some numbers $\alpha_k$ and $\beta_k$.
Consider the equation
\begin{equation}\label{l2product}
G^{a,b}(P_h(w), P_h(z)) =  \sum_{k=1}^n \frac{1}{n}\alpha_k \xi_k
\end{equation}
for some numbers $\xi_k$.
Then
\begin{align}
\beta_k &= \frac{1}{b}\xi_k + n h_{k+1} \sin{\Delta_{k+1}},  \label{grad_discrete_1}
\end{align}
where the sequence $h_k$ satisfies
\begin{align}\label{equation_hk}
\tfrac{1}{n}  \xi_{k-1} \sin{\Delta_k} = (a + a \cos^2{\Delta_k}) h_k - a\cos{\Delta_k} h_{k-1} - a\cos{\Delta_{k+1}} h_{k+1}.
\end{align}
\end{proposition}

\begin{remark}
Note that equation~\eqref{equation_hk} can be written as
\begin{equation*}
\tfrac{1}{a}
\left(\begin{smallmatrix}
\sin \Delta_1 \xi_n\\ \sin \Delta_2 \xi_1\\ \sin \Delta_3\xi_2\\  \vdots \\ \sin\Delta_{n-2}\xi_{n-3}\\ \sin \Delta_{n-1}\xi_{n-2} \\ \sin\Delta_n\xi_{n-1}
\end{smallmatrix}\right)
=
\operatorname{M}
\left(\begin{smallmatrix}
h_1 \\ h_2 \\ h_3\\  \vdots \\ h_{n-2}\\h_{n-1}\\ h_n
\end{smallmatrix}\right),
\end{equation*}
where  $\operatorname{M}$ is the following cyclic tridiagonal matrix
\begin{equation}\label{M}
\operatorname{M} =
n \left(\begin{smallmatrix}
\delta_1& t_2 & 0&  \cdots &0 & 0& t_1\\
t_2 & \delta_2 & t_3 & \cdots & 0 & 0& 0 \\
0 & t_3 & \delta_3 & \cdots & 0 & 0 & 0\\
\vdots & \vdots & \vdots & \cdots & \vdots & \vdots & \vdots\\
0 & 0 & 0 & \cdots & t_{n-1} & \delta_{n-1} & t_n\\
t_1& 0 & 0 & \cdots & 0 & t_n & \delta_n
\end{smallmatrix}\right).
\end{equation}
where $\delta_k = 1+ \cos^2(\Delta_k)$ and where $t_k = -\cos(\Delta_k)$. Note that again, $M$ is strictly dominant as soon as $-\frac{3}{4}<\cos\Delta_{k+1}$ (see remark \ref{stability}).
\end{remark}

\begin{proof}
First of all, we have
$ G^{a,b}(P_h(w),P_h(z)) = G^{a,b}(w, P_h(z))$, since the projection $P_h$ is orthogonal with respect to $G^{a, b}$.
Since the vector field $z$ satisfies  $n(z_{k+1}-z_k) = \beta_k \operatorname{n}_k$, by Theorem~\ref{horizontal_discrete_projection_thm}, its horizontal projection reads
$$
P_h(z) = z_k -  h_k\operatorname{v}_k,
$$
where
$h_k$ is the solution of
\begin{equation}\label{midstep}
\tfrac{b}{n}  \beta_{k-1} \sin{\Delta_k} - b h_k \sin^2{\Delta_k} = a\big(h_k + \cos^2{\Delta_k} h_k - \cos{\Delta_k} h_{k-1} - \cos{\Delta_{k+1}} h_{k+1}\big).
\end{equation}
Using the expression of the $G^{a, b}$-inner product given in \eqref{elastic_discrete_product2}, it follows that
$$
\begin{array}{rl}
 G^{a,b}(w, P_h(z)) &= n\sum_{k=1}^n \frac{b}{n}\alpha_k\langle (z_{k+1} -  h_{k+1}\operatorname{v}_{k+1})-(z_k -  h_k\operatorname{v}_k), \textrm{n}_k\rangle\\  & \\&= n\sum_{k=1}^n b\frac{\alpha_k}{n} (\frac{\beta_k}{n} - h_{k+1} \sin\Delta_{k+1}),
 \end{array}
$$
where we have used $n(z_{k+1} - z_{k}) = \beta_k \operatorname{n}_k$ and $\langle \operatorname{v}_{k+1}, \operatorname{n}_k\rangle = \sin \Delta_{k+1}$.
Comparing with equation~\eqref{l2product}, it follows that
$$
\frac{1}{n} \xi_k =  \frac{b}{n} (\beta_k - n h_{k+1}\sin\Delta_{k+1}).
$$
Therefore equation~\eqref{midstep} reads
$$
\frac{1}{n}\sin{\Delta_k} \xi_{k-1}  = a \big(h_k + \cos^2{\Delta_k} h_k - \cos{\Delta_k} h_{k-1} - \cos{\Delta_{k+1}} h_{k+1}\big).
$$
\end{proof}

Let us summarize the previous results in the following Theorem.
\begin{theorem}
Suppose we have a family of curves $\gamma_k(\varepsilon, t)$ depending on time and joining fixed curves $\gamma_{1,k}$ and $\gamma_{2,k}$
(which is to say that $\gamma_k(\varepsilon, 0) = \gamma_{1,k}$ and $\gamma_k(\varepsilon, T) = \gamma_{2,k}$ for all $\varepsilon$ and $k$).
Then the derivative of the energy functional $E$ associated with the quotient elastic metric $G^{a,b}$ reads
$$
\frac{dE}{d\varepsilon}(0) = \int_{0}^{T} G^{a,b}({\gamma}_\varepsilon,\grad E(\gamma)) dt,
$$
where $\grad E(\gamma) = (z_k : 1\leq k \leq n)$ is the solution of $n(z_{k+1} - z_k) = \beta_k \operatorname{n}_k$ with $\beta_{k}$ solving \eqref{grad_discrete_1} for $\xi_k $
defined by \eqref{xidef}. Since we consider curves modulo translation, we can take $z_0 = 0$. The projection of $\grad E(\gamma)$ on the manifold of closed curves reads $$\left(z_k -\frac{1}{n}\sum_{k=1}^n \beta_k \operatorname{n}_k : 1\leq k \leq n\right).$$
\end{theorem}

\section{Two-boundary problem}\label{implementation}
\subsection{Algorithms for the two-boundary problem}\label{algo}

Given two shapes in the plane, solving the two-boundary problem consists in finding a geodesic (if it exists!) having these shapes as endpoints. A geodesic is a path that is locally length-minimizing.
Using the exact expression of the gradient of the energy functional, we can obtain approximations of geodesics by a path-straightening method. This method relates to the fact that critical points of the energy are geodesics, and it consists of straightening an initial path between two given shapes in the plane by following the opposite of the gradient flow of the energy functional (see Section~\ref{algo}, Algorithm~\ref{path_straightening}).
The algorithm for the computation of the gradient of the energy functional, based on the computation given in previous sections,  is detailed in Section~\ref{algo}, Algorithm~\ref{compute_gradient}.
Of course the efficiency of the path-straightening method depends greatly on the landscape created by the energy functional on the space of paths connecting two shapes, and this landscape in turns varies with the parameters $a$ and $b$ of the elastic metric.
In Section~\ref{Energy_landscape}, we illustrate some aspects of this dependence. In all the numerics presented in the paper we used $100$ points for each curve. 
\begin{figure}[!ht]
 		\centering
		\includegraphics[width = 12cm]{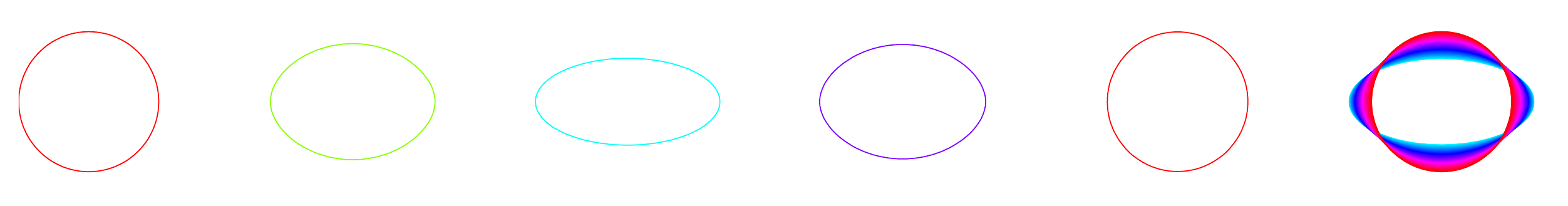}
\caption{Toy example: initial path joining a circle to the same circle via an ellipse. The 5 first shapes at the left correspond to the path at time $t = 0$, $t = 0.25$, $t = 0.5$, $t = 0.75$ and $t = 1$.
The right picture shows the entire path, with color varying from red ($t=0$) to blue ($t = 0.5$) to red again ($t=1$).
}
 		\label{initial_path_circle_ellipse}		
\end{figure}
\begin{figure}[!ht]
 		\centering
		\includegraphics[height = 3.8cm]{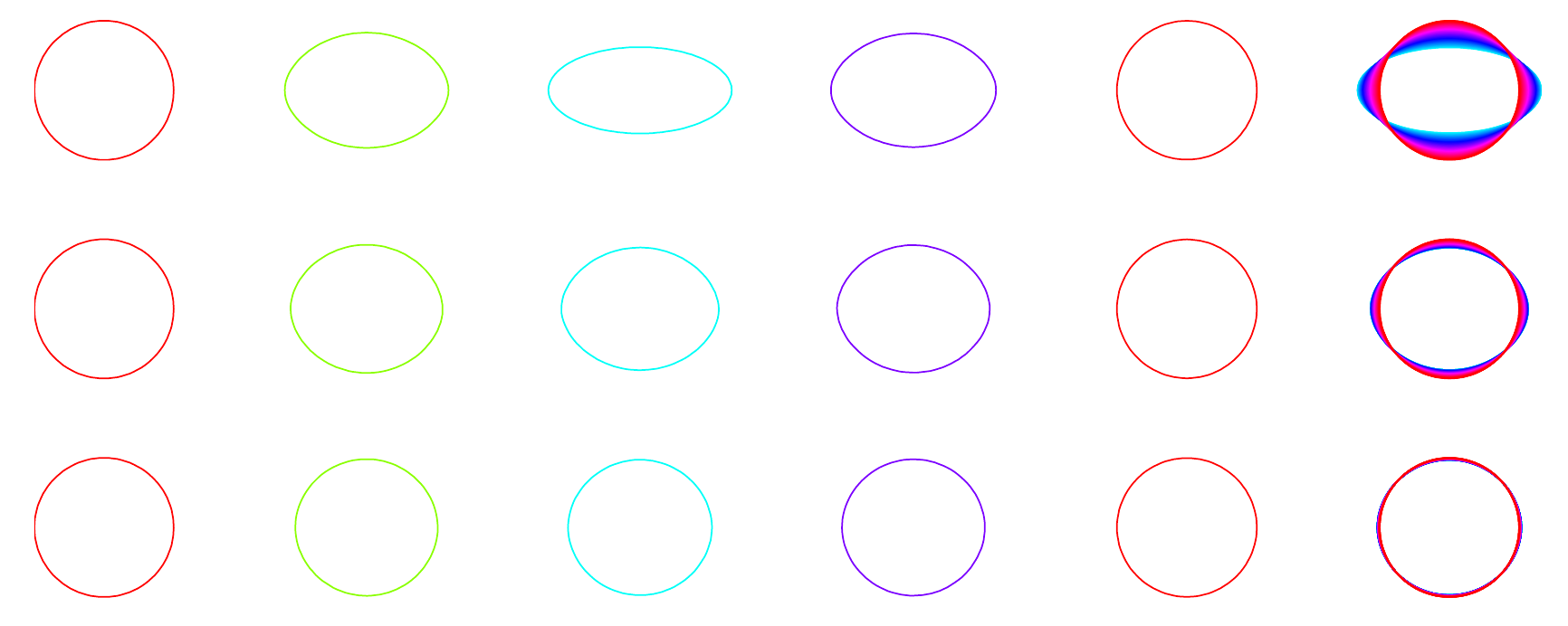}
\fontsize{10}{12}\selectfont
		\includegraphics[height = 3.8cm]{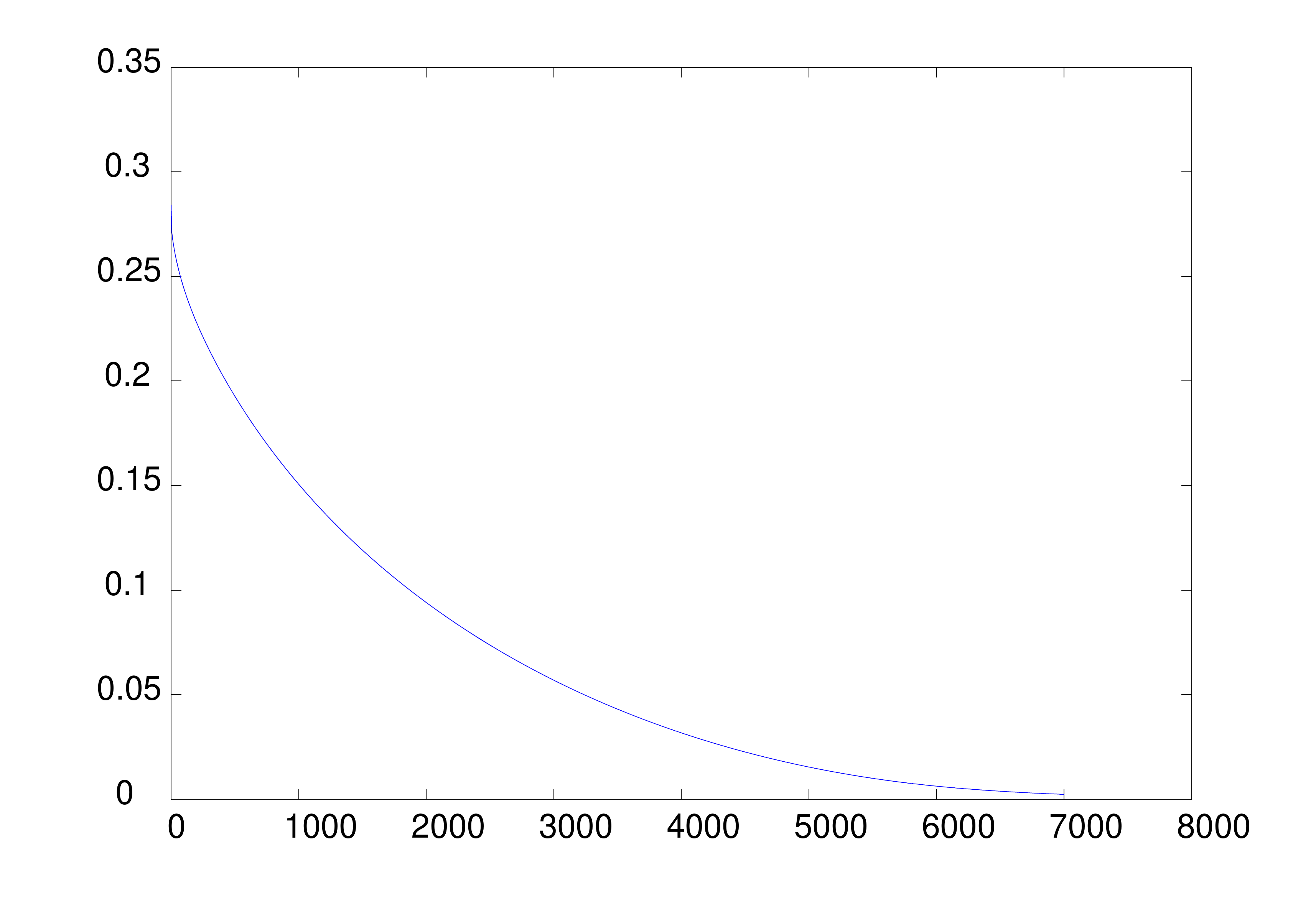}
		 		\caption{Straightening of the path illustrated in Fig.~\ref{initial_path_circle_ellipse}, with $a=100$ and $b=1$. The first line corresponds to the initial path, the second line to the path after 3500 iterations, and the third line corresponds to the path after 7000 iterations. Underneath, the evolution of the energy with respect to the number of iterations is depicted.}
 		\label{circle_ellipse_c100}		
 \normalsize
\end{figure}

\begin{algorithm}
\begin{small}
\KwIn{\begin{enumerate}
\item An initial shape $\gamma_1$ given by the positions $\gamma_{k,1}, 1\leq k \leq n$ of $n$ points in $\mathbb{R}^2$,
\item A final shape $\gamma_2$ given by the positions $\gamma_{k,2}, 1\leq k \leq n$ of $n$ points in $\mathbb{R}^2$.
\end{enumerate}
}
\KwOut{
An (approximation of a) geodesic between $\gamma_1 $ and $\gamma_2$ under the quotient elastic metric $G^{a, b}$, given by the positions $\gamma_{k}(t), 1\leq k \leq n$ of $n$ points in $\mathbb{R}^2$, with $\gamma_k(0) = \gamma_{k,1}$ and $\gamma_k(1) = \gamma_{k, 2}.$
}
\vspace{0.3cm}
\textbf{Algorithm~1:} Initialize $\gamma_k(t)$ by  a path connecting $\gamma_1$ to $\gamma_2$.
\begin{enumerate}
\item compute $\grad E(\gamma)$ using Algorithm~\ref{compute_gradient}.
\item while $\grad E(\gamma)<10^{-3}$ do
\begin{enumerate}
\vspace{0.05cm}
\item $\gamma_k(t) \leftarrow \gamma_k(t) - \delta \grad E(\gamma)$ where $\delta$ is a small parameter to be adjusted (we\break  used $\delta = 10^{-9}$).
\vspace{0.1cm}
\item Compute the length $L(\gamma)$ of $\gamma_k(t)$ and do $\gamma_k(t) \leftarrow \gamma_k(t)/L(\gamma)$. 
\end{enumerate}
\end{enumerate}
\end{small}\vspace*{4pt}
\caption{Algorithm for the path-straightening method}
\label{path_straightening}
\end{algorithm}

\begin{algorithm}
\begin{small}
\KwIn{positions $\gamma_k(t), 1\leq k \leq n$ of $n$ points in $\mathbb{R}^2$ depending on time $t\in I$.
}
\vspace{0.1cm}
\KwOut{$n$ vectors $z_k = \grad E_k(t), 1\leq k \leq n$ in $\mathbb{R}^3$, depending on time $t\in I$, corresponding to the values of the gradient of the $G^{a, b}$-energy of $\gamma_k(t)$.}
\vspace{0.1cm}
\textbf{Algorithm~2:}
\begin{enumerate}
\item compute  $
\theta_k$ defined by
$$(\cos\theta_k(t), \sin\theta_k(t)) = n(\gamma_{k+1}(t) - \gamma_{k}(t))/\lvert\gamma_{k+1}(t) - \gamma_{k}(t))\rvert,$$
 then compute $\dot{\theta}_k$ and $\Delta_k =\theta_{k+1}-\theta_k$.
\item define $T$ as in equation~\eqref{T} and compute $(m_k, 1\leq k\leq n)$ defined by:
$$
\operatorname{T}\left(\begin{smallmatrix}
m_1 \\m_2  \\m_3 \\ \vdots \\m_{n-1} \\m_{n}
\end{smallmatrix}\right) = \frac{b}{n}\left(\begin{smallmatrix}
\dot{\theta}_n  \sin \Delta_1  \\  \dot{\theta}_1 \sin \Delta_2 \\ \dot{\theta}_2 \sin \Delta_3 \\ \vdots \\  \dot{\theta}_{n-2} \sin \Delta_{n-1}\\ \dot{\theta}_{n-1} \sin \Delta_n
\end{smallmatrix}\right).
$$
\item compute $\ddot{\theta}_k$ and $\dot{\Delta}_k$ as well as
$$
\begin{array}{rl}R_k  = & \Big( \frac{b}{n}\sin{\Delta_k} \ddot{\theta}_{k-1} + \frac{b}{n} \cos{\Delta_k} \dot{\theta}_{k-1} \dot{\Delta}_k
+ 2(a-b) \sin{\Delta_k}\cos{\Delta_k} m_k \dot{\Delta}_k \\ &
- a \sin{\Delta_k} m_{k-1} \dot{\Delta}_k - a \sin{\Delta_{k+1}} m_{k+1} \dot{\Delta}_{k+1}\Big).\end{array}
$$
\item compute $\dot{m}_k$ defined by equation \eqref{mdotformula}: $\operatorname{T} \dot{m} =  R.$
\item compute $\xi_k $ defined by equation~\eqref{xidef}:
\begin{multline*}
\xi_k = -b \ddot{\theta}_k +  b n (\dot{m}_{k+1} \sin{\Delta_{k+1}}  + m_{k+1} \cos \Delta_{k+1} \dot{\theta}_{k+1} - m_k \cos \Delta_k \dot{\theta}_{k-1})\\
+ n^2 (b-a) (m_k^2 \sin{\Delta_k} \cos{\Delta_k} -  m_{k+1}^2 \sin{\Delta_{k+1}} \cos{\Delta_{k+1}}) \\
+ a n^2 m_k (m_{k-1} \sin{\Delta_k} - m_{k+1} \sin{\Delta_{k+1}}).
\end{multline*}
\item define matrix $\operatorname{M}$ by equation~\eqref{M} and compute $h_k$  defined by:
$$
\operatorname{M}\left(\begin{smallmatrix}
h_1 \\ h_2 \\ h_3\\  \vdots \\ h_{n-2}\\h_{n-1}\\ h_n
\end{smallmatrix}\right) =
\tfrac{1}{a}
\left(\begin{smallmatrix}
\sin \Delta_1 \xi_n\\ \sin \Delta_2 \xi_1\\ \sin \Delta_3\xi_2\\  \vdots \\ \sin\Delta_{n-2}\xi_{n-3}\\ \sin \Delta_{n-1}\xi_{n-2} \\ \sin\Delta_n\xi_{n-1}
\end{smallmatrix}\right).
$$
\item compute $\beta_k$ defined by equation~\eqref{grad_discrete_1}:
$
\beta_k =\frac{1}{b} \xi_k + n h_{k+1} \sin{\Delta_{k+1}}.
$
\item compute $z_k$ defined by $z_1=0$ and $z_{k+1}  =  z_k + \frac{1}{n}\beta_k \operatorname{n}_k-\frac{1}{n}\sum_{k=1}^n\beta_k \operatorname{n}_k$.

\end{enumerate}
\end{small}\vspace*{4pt}
\caption{Algorithm for the computation of the gradient of the energy functional}
\label{compute_gradient}
\end{algorithm}

\subsection{Energy landscape}\label{Energy_landscape}

In order to experience the range of convergence of the path-straightening algorithm, we first start with a toy example, namely we start with an initial path joining a circle to the same circle but passing by an ellipse in the middle of the path. This path is illustrated in Fig.~\ref{initial_path_circle_ellipse}, where the middle ellipse may by replaced by an ellipse with different eccentricity.
Starting with this initial path, we expect the path-straightening method to straighten it into the constant path containing only circles, which is a geodesic. However, this will happen only if the initial path is in the attraction basin of the constant path, in the sense of dynamical systems, i.e., if the initial path is close enough to the constant geodesic. This in turn will depend on the value of the parameter $a/b$ of the elastic metric. In particular the same path can be in the attraction basin of the constant path for some value of $a/b$ and outside of it for some other value of the parameter.
\begin{figure}[!ht]
 		\centering
		\includegraphics[width = 12cm]{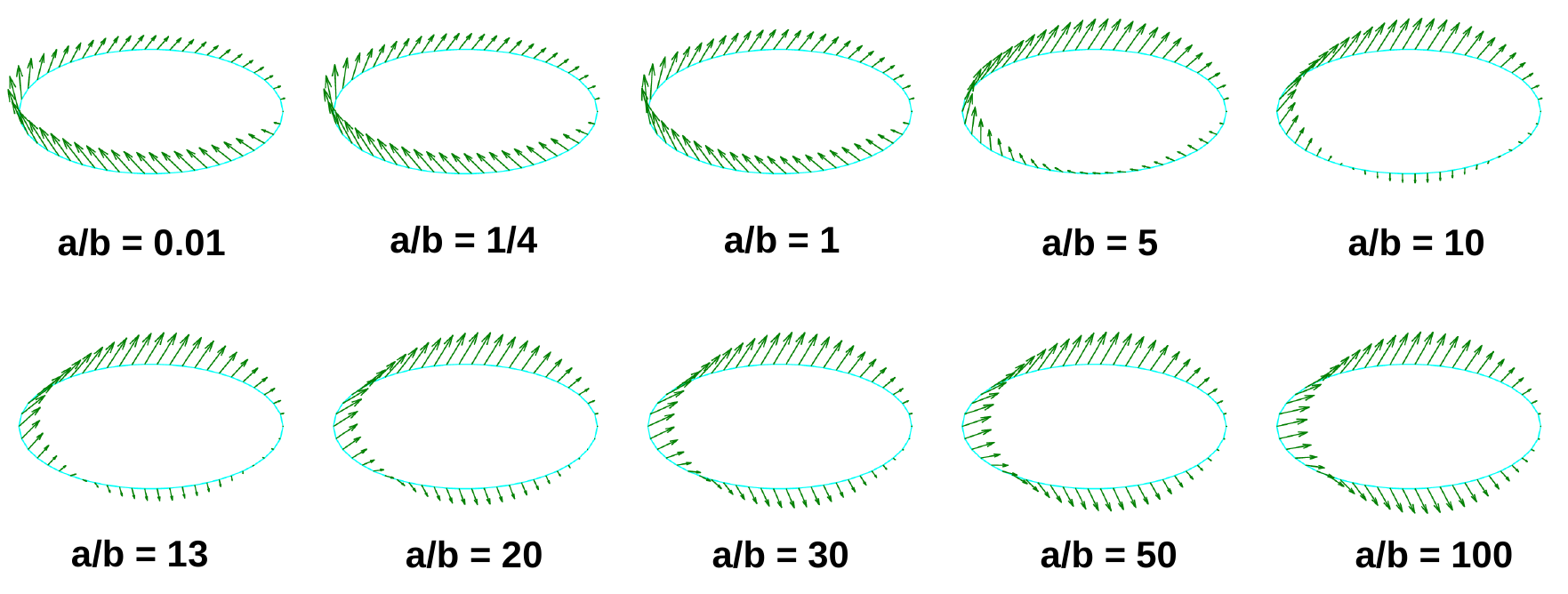}

 		\caption{Negative gradient of the energy functional at the middle of the path depicted in Fig.~\ref{initial_path_circle_ellipse} for $b=1$ and different values of the parameter $a/b$. }
 		\label{change_grad}		
\end{figure}
In order to have a better idea when the path-straightening method will converge, we plot in Fig.~\ref{change_grad} the opposite of the gradient of the energy functional at the middle of the path for different values of the parameter $a/b$. In this figure, the magnitude of the gradient is rescaled, hence the only important information is the directions taken by the vector field. For $a/b = 100$, the opposite of the gradient is the vector field that one expects for turning the ellipse into a circle. On the contrary, for $a/b = 0.01$, the opposite of the gradient is not bowing the ellipse. In other words, one can conjecture that the initial path depicted in Fig.~\ref{initial_path_circle_ellipse} is in the attraction basin of the constant path for $a/b = 100$, but not for $a/b = 0.01$. This is indeed what is happening, the path-straightening algorithm applied to the path of Fig.~\ref{initial_path_circle_ellipse} converges for $a/b = 100$ (see Fig.~\ref{circle_ellipse_c100}) but diverges/ for $a/b = 0.01$.

\begin{figure}[!ht]
 		\centering

		\includegraphics[width = 12cm]{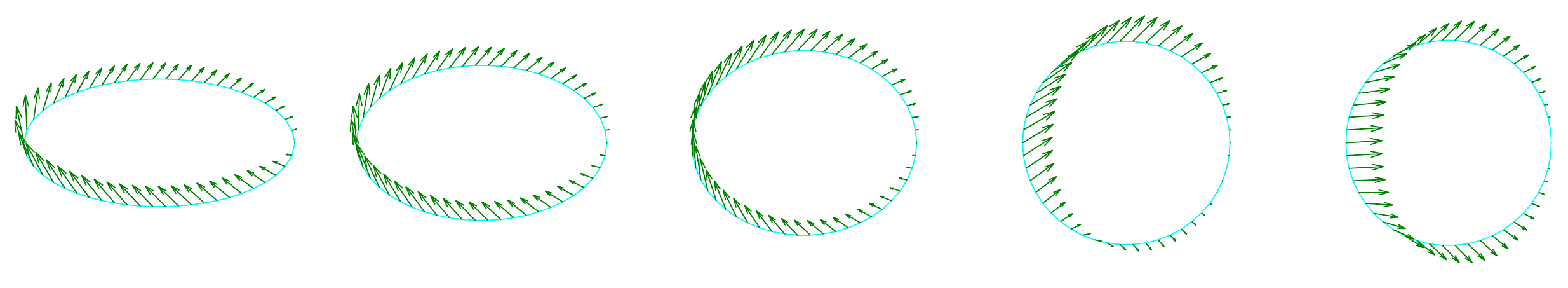}
		\includegraphics[width = 12cm]{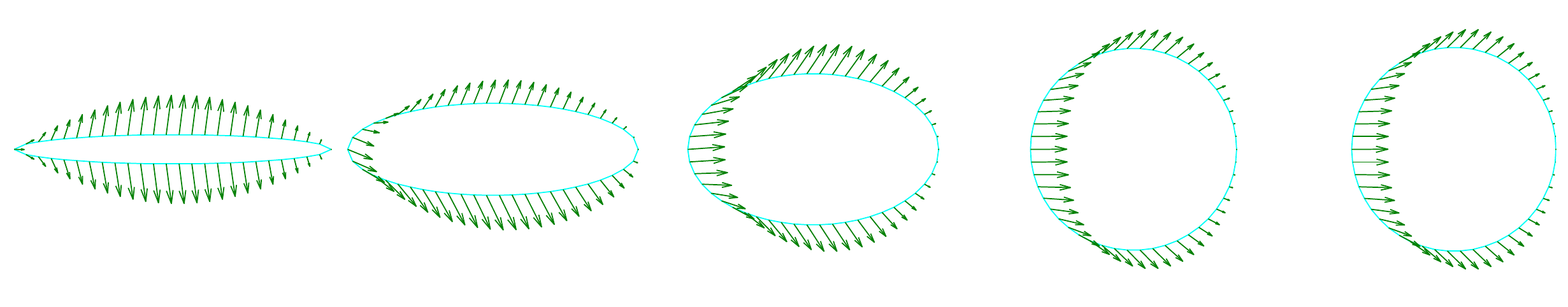}
 		\caption{Negative gradient of the energy functional at the middle of the path connecting a circle to the same circle via an ellipse for different values of the eccentricity of the middle ellipse. The first line corresponds to the values of parameters $a =0.01$ and $b=1$. The second line corresponds to $a = 100$ and $b=1$. }
 		\label{grad_lanscape}		
\end{figure}
To have an idea of the attraction basin of the constant geodesic for $a/b = 0.01$, one can vary the eccentricity of the middle ellipse in the initial path. Recall that the ellipse eccentricity is defined as $e = \sqrt{1-c^2/d^2}$ with $c$ the semi-minor axis and $d$ the semi-major axis. In Fig.~\ref{grad_lanscape}, we have depicted the gradient of the energy functional at the middle of the initial path for different values of the middle ellipse's eccentricity.  The first line corresponds to $a/b = 0.01$. From left to right the eccentricity of the ellipse at the middle of the path takes the values $0.8844$, $0.7882$, $0.5750$, $0.1980$ and $0.0632$. One sees a change in the vector field between the third and fourth picture: only when the middle ellipse is nearly a circle will  the path-straightening algorithm converge for the value $a/b = 0.01$. In comparison, the second line corresponds to $a/b = 100$. From left to right the eccentricity of the ellipse at the middle of the path takes the values $0.9963$, $0.95$, $0.8$, $0.1980$ and $0.0632$. In this case, the opposite of the gradient is bowing the ellipse even if the ellipse is very far from a circle.

Another aspect of the gradient in this toy example is that it is localized at the middle shape as is illustrated in Fig.~\ref{grad_localization}. In this picture the gradient is scaled uniformly. One sees that the gradient is nearly zero except at the middle shape. This is clearly a disadvantage for the path-straightening method since after one iteration of algorithm~\ref{path_straightening}, only the middle shape is significantly changed. This localization of the gradient imposes a small step size in order to avoid discontinuities in the path around the middle shape.
\begin{figure}[!ht]
 		\centering
        \includegraphics[width = 12cm]{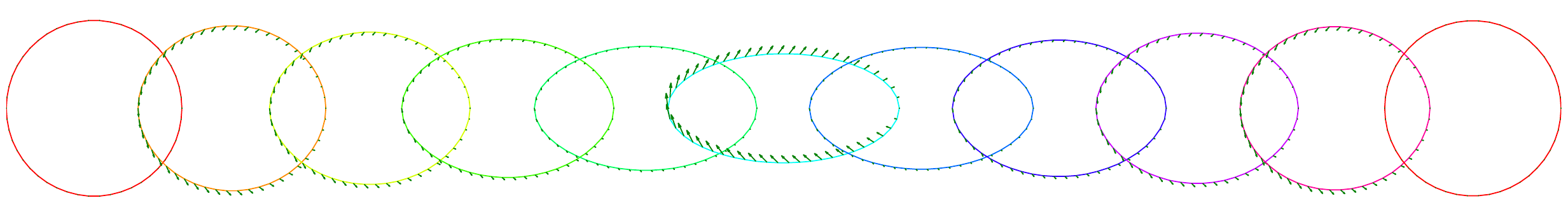}
        \includegraphics[width = 12cm]{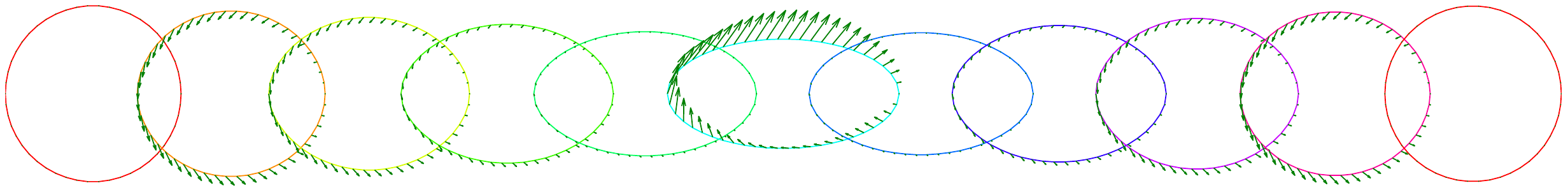}
        \includegraphics[width = 12cm]{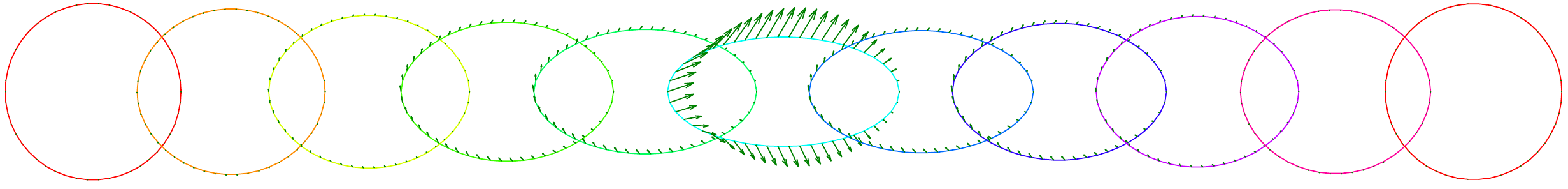}
 		\caption{Negative gradient of the energy functional along the path depicted in Fig.~\ref{initial_path_circle_ellipse} for $a=1$ (upper line), $a = 5$ (middle line) and $a = 50$ (lower line) and $b = 1$. }
 		\label{grad_localization}		
\end{figure}

In Fig.~\ref{carrel}, we show a $2$-parameter family of variations of a circle. The middle horizontal line corresponds to the deformation of the circle into an ellipse, and can be thought of as stretching the circle by pulling or pushing it to opposite circle points. In comparison, the middle vertical column corresponds to the deformation of the circle into a square and can be thought of as bending the circle at four corners. We built a $2$-parameter family of deformations of the constant path connecting a circle to itself by interpolating smoothly from the circle to one of these shapes at the middle of the path and back to the circle.
In Fig.~\ref{Ec}, the energy plots of the $2$-parameter family of paths obtained this way are depicted for $a= 0.01$, $b =1$ (left upper picture and nearly flat piece in the lower picture), and for $a = 100$, $b =1$ (right upper picture, and curved piece in the lower picture). One sees that, for the elastic metric with $a=0.01$, $b = 1$, both directions of deformation---turning a circle into an ellipse and turning a circle into a square---have the same energy amplitude. On the contrary, for the elastic metric with $a=100$ and $b=1$, one needs a lot more energy to deform a circle into an ellipse than to deform a circle into a square, i.e., stretching is predominant.
\begin{figure}[!ht]
 		\centering

		\includegraphics[width = 8cm]{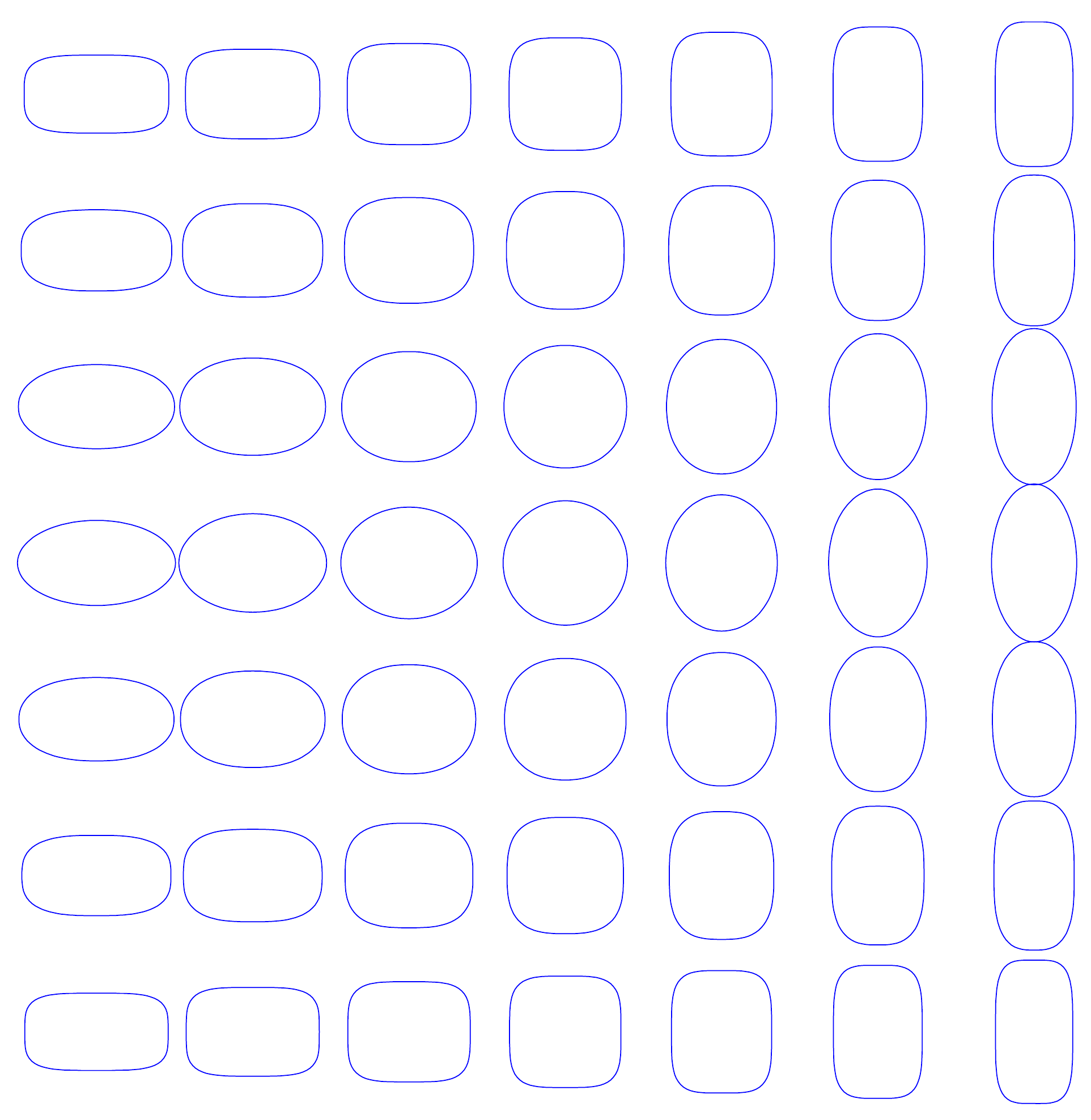}
 		\caption{$2$-parameter family of variations of the middle shape of a path connecting a circle to the same circle}
 		\label{carrel}		
\end{figure}
\begin{figure}[!ht]
 		\centering

		\includegraphics[width = 5.5cm]{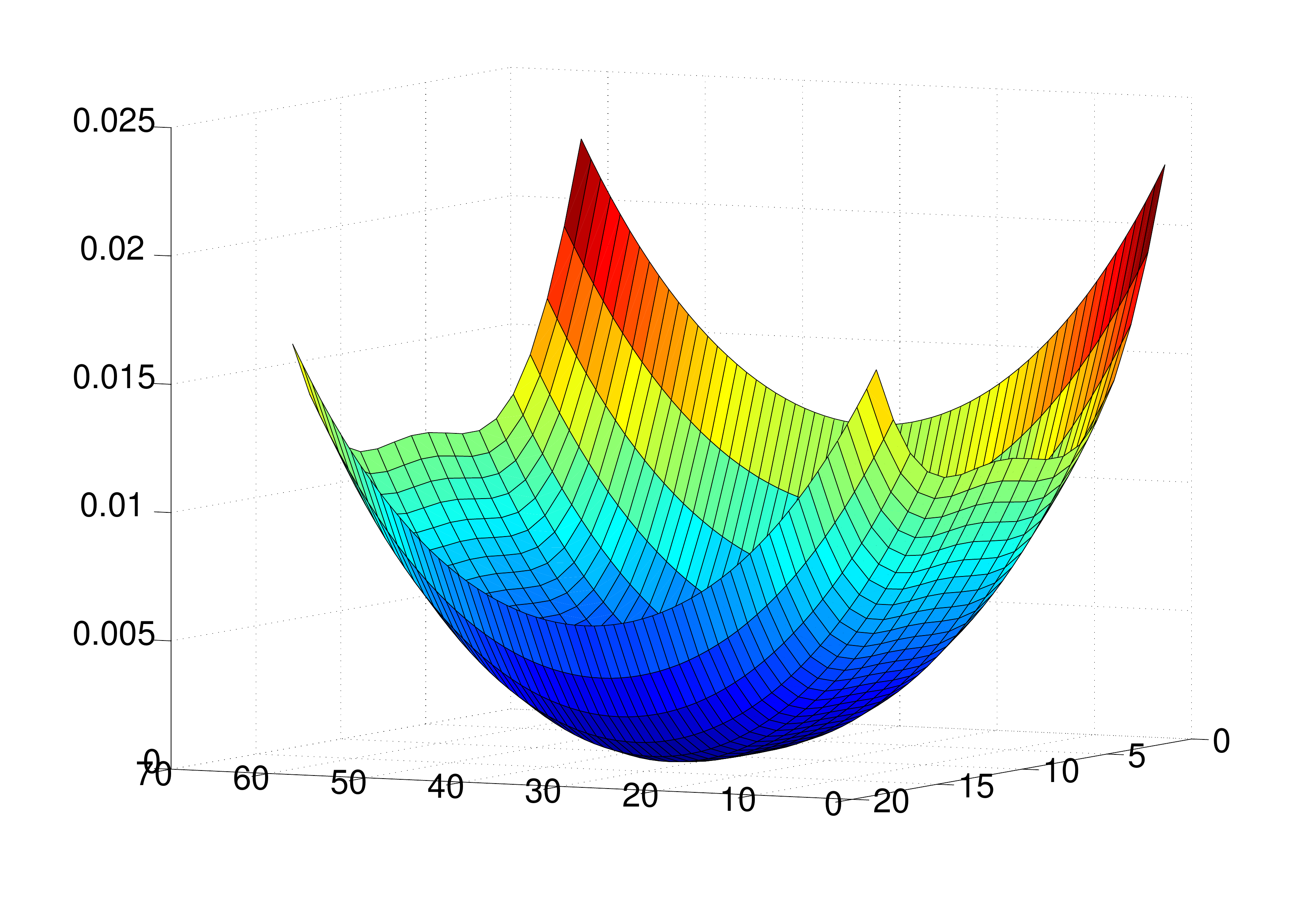}
		\includegraphics[width  = 5.5cm]{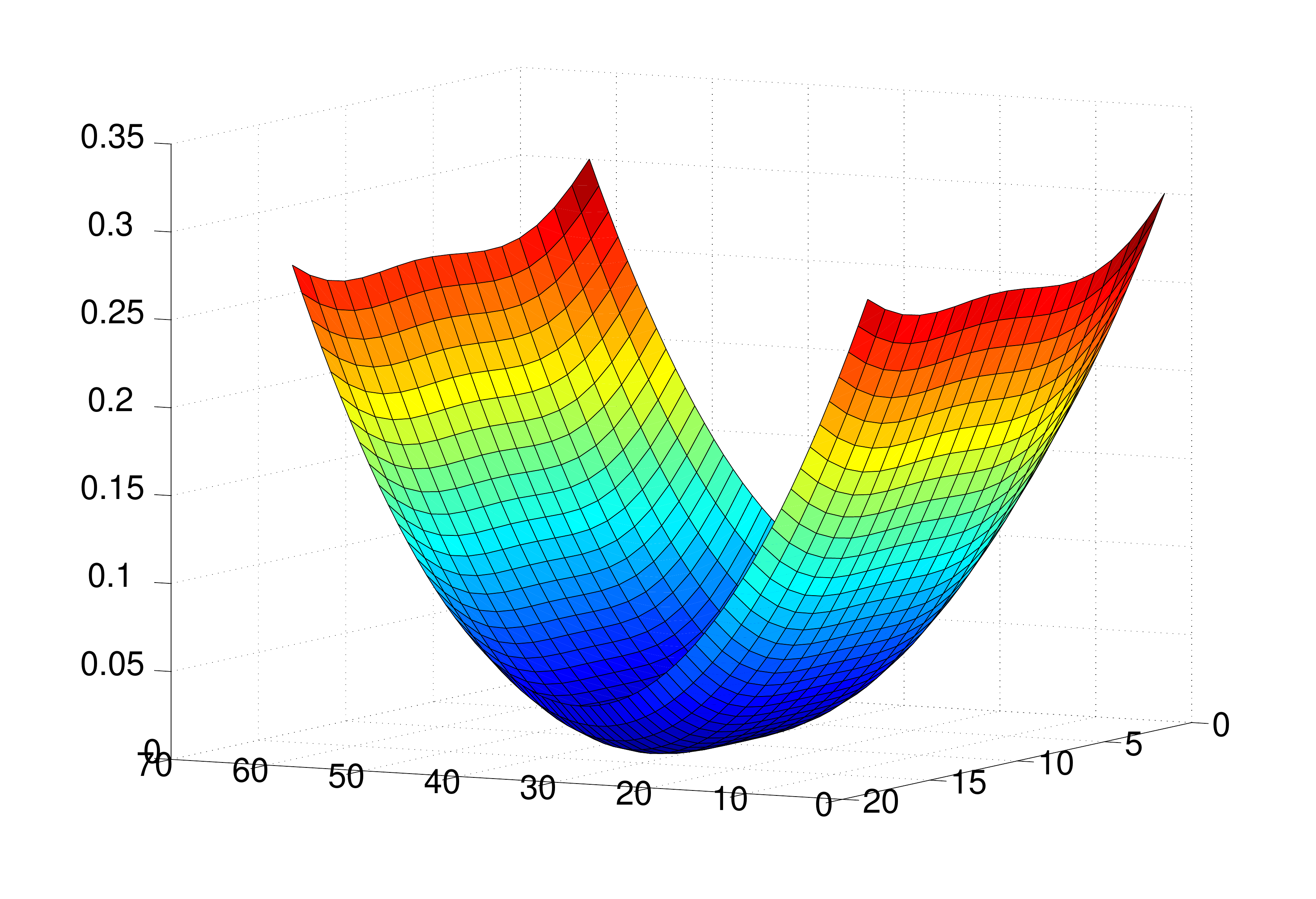}
                 \includegraphics[width  = 5.5cm]{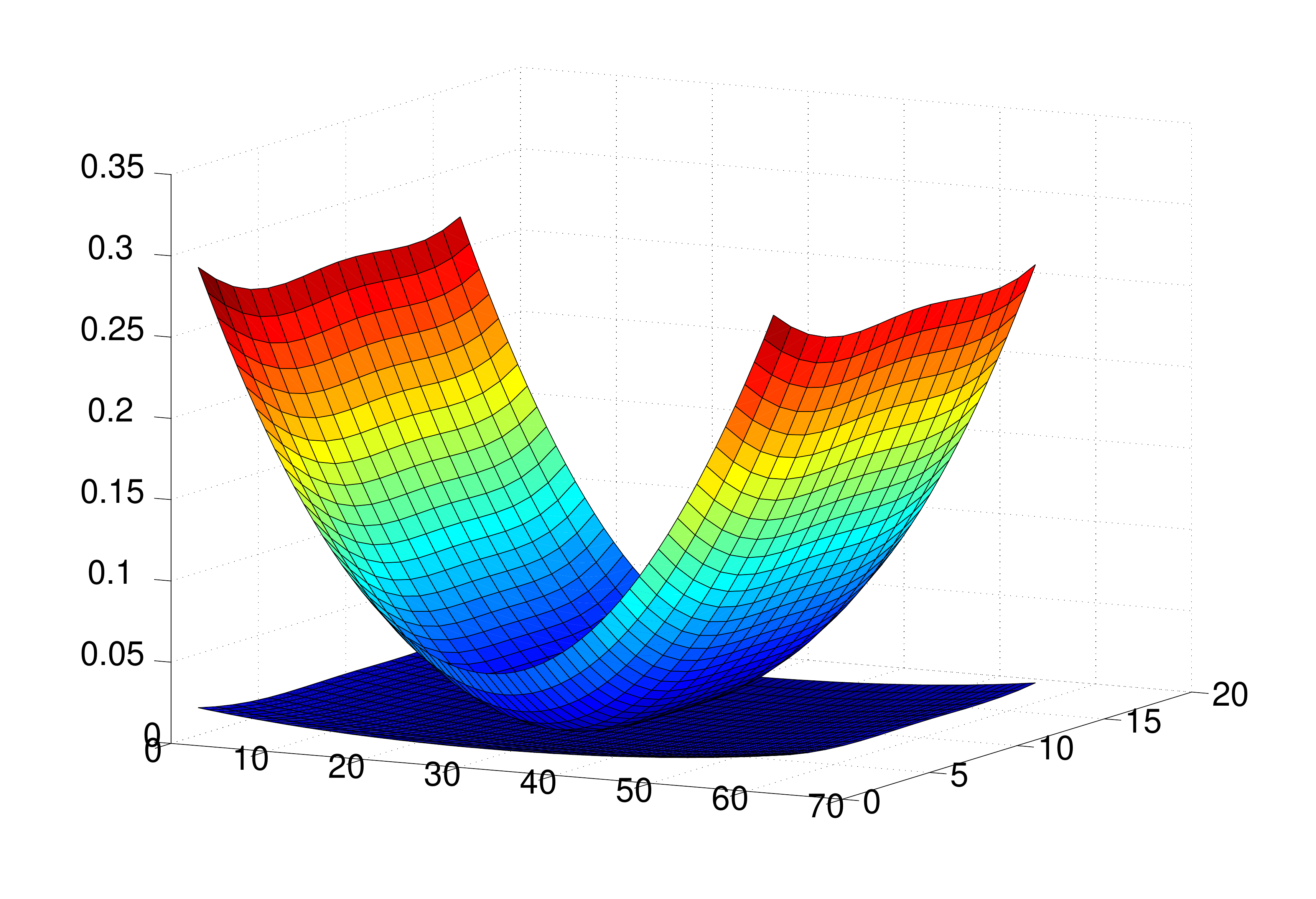}
 		\caption{Energy functional for the $2$-parameter family of paths whose middle shape is one of the shapes depicted in Fig.~\ref{carrel}. The left upper picture corresponds to $a= 0.01$, $b =1$ and the right upper  picture to $a = 100$, $b =1$. The lower picture shows the plots of both energy functionals with equal axis.}
 		\label{Ec}		
\end{figure}

Finally we consider in Fig.~\ref{geo_approxim} the problem of finding a geodesic from a Mickey Mouse hand to the same hand with a finger missing. The first line is obtained by taking the linear interpolation of the hands, when both hands are parameterized by arc-length. The second line is obtained by first taking the linear interpolation of the hands and then parameterizing each shape of the path by arc-length. The second path serves as initial path for the path-straightening method. The third line (resp. the fourth line, resp. the last line) corresponds to the path of minimal energy that we were able to find for $a = 0.01$, $b =1$ (resp. $a = 0.25$, $b =1$, resp. $a =100$, $b=1$), but the path-straightening algorithm is struggling in all cases. 
Note the different shapes of the growing finger when the parameters are changed. The energy of all these paths, for the different values of the parameters, is given in Tab.~\ref{energies}.
\begin{figure}[!ht]
 		\centering
		\includegraphics[width = 12cm]{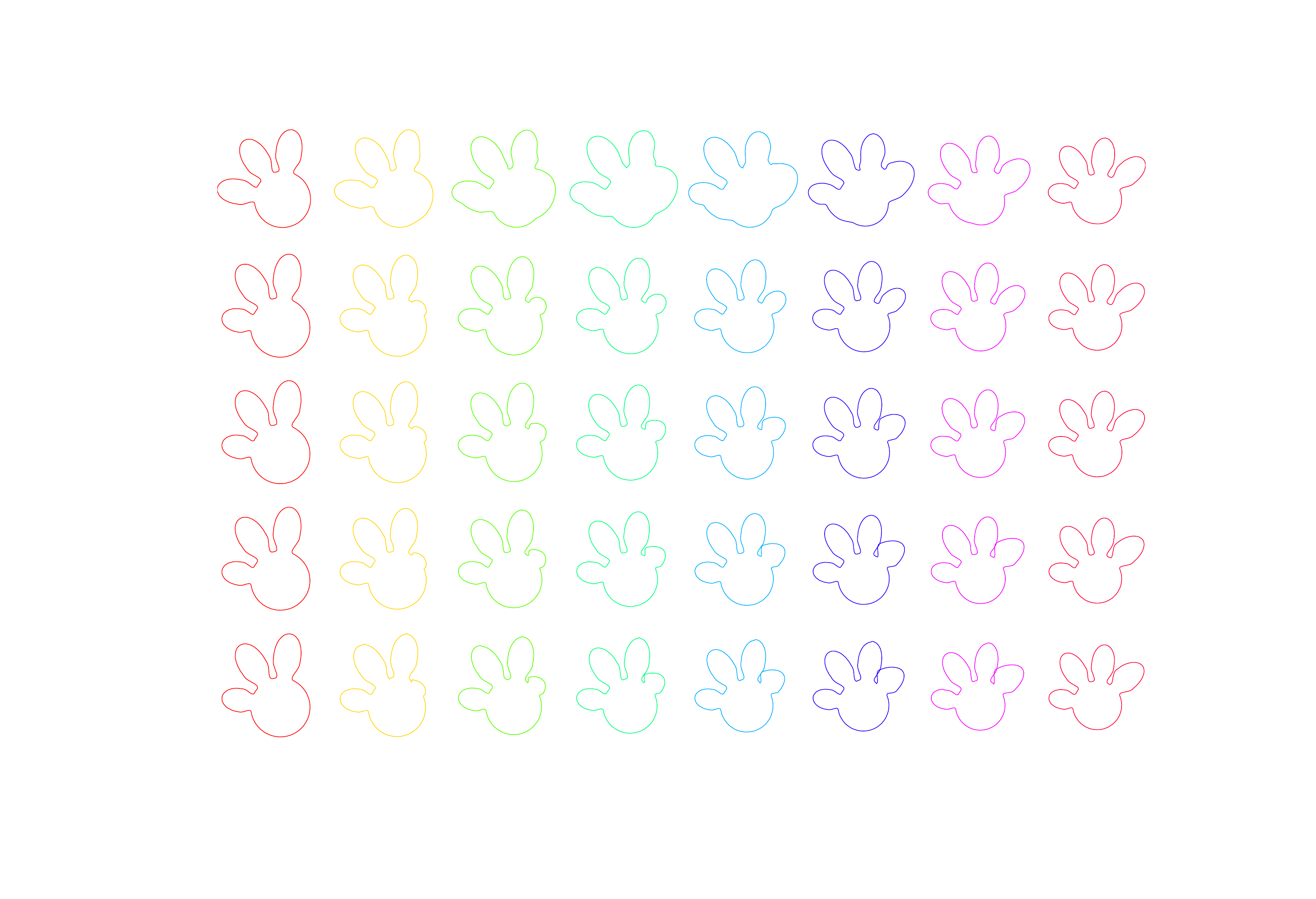}
 		\caption{Different paths connecting a Mickey Mouse hand to the same hand with a missing finger}
 		\label{geo_approxim}	
\end{figure}

\begin{table}
\begin{center}
\begin{tabular}{|c|c|c|c|c|c|}
\hline
\small{parameters} & \small{lin. interpol. 1} &  \small{lin. interpol. 2} & \small{path 3} & \small{path 4} & \small{path 5}\tabularnewline
\hline
\hline
$a=0.01$, $b = 1$ & 32.3749 & 27.45 & \textcolor{red}{25.3975} & 26.2504 & 28.3768
\tabularnewline
\hline
$a = 0.25$, $b =1$ & 63.1326 & 52.4110 & 47.8818 & \textcolor{red}{47.5037} & 48.2284
\tabularnewline
\hline
$a =100$, $b=1$ & 77.6407 & 66.6800 & 63.4840 & 60.9704 & \textcolor{red}{57.4557}
\tabularnewline
\hline
\end{tabular}
\end{center}\vspace*{6pt}
\caption{Energy of the paths depicted in Fig.~\ref{geo_approxim}.}
\label{energies}\vspace*{-6pt}
\end{table}

\section*{Conclusion}

In this paper, we studied the pull-back of the quotient elastic metrics to the space of arc-length parameterized plane curves of fixed length. We computed, for all values of the parameters, the exact energy functional as well as its gradient. These computations allowed us to illustrate how these metrics  behave with respect to stretching and bending. In particular, we showed that even for  small values of  $a/b$, stretching and bending have contributions of the same order of magnitude to the energy, a fact that may be surprising in regard to the expression of the elastic metric on parameterized curves. On the other hand, for large values of $a/b$, stretching has a predominant cost to the energy, as expected. This implies that the energy landscape is  steeper for big values of $a/b$ in the sense that some deformations are preferred, a property that facilitates convergence of a path-straightening algorithm.

\section*{Acknowledgments}

This work emerged from discussions held during the program 2015 ``Infinite-dimensional Riemannian geometry with applications to image matching and shape analysis," at Erwin Schr\"odinger Institute, Vienna, Austria, and the authors would like to thank Eric Klassen and Anuj Srivastava for fruitful interactions, as well as the organizers of this conference for providing such stimulating working conditions. This note was written during a visit of A.B. Tumpach
at the Pauli Institute, Vienna, Austria. The programs illustrated in this paper were implemented using a Matlab license of the University of Vienna. This work was also supported in part by the Labex CEMPI  (ANR-11-LABX-0007-01). Furthermore this work was partially supported by a grant from the Simons Foundation (\#318969 to Stephen C. Preston). The authors would like to thank the anonymous referees for reading carefully the manuscript and helping improve the paper by their pertinent suggestions.

\providecommand{\href}[2]{#2}
\providecommand{\arxiv}[1]{\href{http://arxiv.org/abs/#1}{arXiv:#1}}
\providecommand{\url}[1]{\texttt{#1}}
\providecommand{\urlprefix}{URL }

\medskip
Received January 2016; revised March 2017.
\medskip

\end{document}